\documentclass[a4paper,11pt]{article}
\usepackage[T1]{fontenc}
\usepackage{lmodern,amsmath,amssymb,amsthm,mathtools,mathrsfs}
\usepackage[margin=25.4mm]{geometry}
\usepackage{microtype,booktabs,array,placeins,cite,longtable,enumitem}
\usepackage[hidelinks,bookmarksnumbered]{hyperref}
\hypersetup{pdftitle={A unified approach to infinite log-concavity of combinatorial sequences},
 pdfauthor={Matthew H. Y. Xie, Candice X. T. Zhang, Philip B. Zhang}}
\allowdisplaybreaks[2]
\newtheorem{theorem}{Theorem}[section]
\newtheorem{lemma}[theorem]{Lemma}
\newtheorem{proposition}[theorem]{Proposition}
\newtheorem{corollary}[theorem]{Corollary}
\theoremstyle{remark}
\newtheorem{remark}[theorem]{Remark}

\numberwithin{equation}{section}
\newcommand{\D}{\mathscr D}
\newcommand{\LL}{\mathcal L}
\newcommand{\Rea}{\operatorname{Re}}
\renewcommand{\Im}{\operatorname{Im}}
\newcommand{\B}{\mathrm B}
\newcommand{\norm}[2]{\lVert #1\rVert_{#2,\infty}}
\newcommand{\pnorm}[2]{\lVert #1\rVert_{#2,1}}

\newcommand{\Log}{\operatorname{Log}}

\title{A unified approach to infinite log-concavity\\of combinatorial sequences}
\author{Matthew H. Y. Xie\textsuperscript{1}\quad
Candice X.~T. Zhang\textsuperscript{1}\quad Philip B. Zhang\textsuperscript{2}\\[6pt]
\small \textsuperscript{1}School of Mathematical Sciences,\\
\small Tianjin University of Technology, Tianjin 300384, P. R. China\\[4pt]
\small \textsuperscript{2}College of Mathematical Sciences\\
\small \& Institute of Mathematics and Interdisciplinary Sciences,\\
\small Tianjin Normal University, Tianjin 300387, P. R. China\\[5pt]
\small \href{mailto:xie@email.tjut.edu.cn}{\texttt{xie@email.tjut.edu.cn}};\quad
\href{mailto:zhangxutong@email.tjut.edu.cn}{\texttt{zhangxutong@email.tjut.edu.cn}}\\
\small \href{mailto:zhang@tjnu.edu.cn}{\texttt{zhang@tjnu.edu.cn}}}

\date{}
\begin{document}
\maketitle
\begin{abstract}
We establish a sufficient condition for infinite log-concavity
and apply it to five families of combinatorial sequences.
We give a new proof of infinite log-concavity for the
Boros--Moll coefficient sequences $(d_k(n))_{k=0}^n$ for
fixed $n$. For the transposed sequences
$(d_\ell(\ell+k))_{k\ge0}$, we prove Zhao's conjecture on
infinite log-concavity for every fixed integer $\ell\ge3$.
For the normalized sequences in Euler's difference table,
we confirm the infinite log-concavity conjecture of
Chen, Gu, Ma and Wang under essential iteration, in which
both endpoints are discarded after each step.
We also confirm a conjecture of Medina, Moll and Rowland
on the infinite log-concavity of the coefficient sequences
of polynomials arising from iterated primitives of
$\log(1+x)$.
Finally, we answer a question of Br\"and\'en and Chasse by
showing that, for positive integers $d$, the sequence
$(k^d)_{k\ge0}$ is infinitely log-concave if and only if $d\ne2$.
For every real $d\ge3$, all its iterates are positive at every
index $k\ge1$. The proofs combine holomorphic estimates with
finite computer-assisted verification.
\end{abstract}

\section{Introduction}\label{sec:introduction}
The purpose of this paper is to establish a unified approach
to infinite log-concavity and to apply it to several families
of combinatorial sequences. Recall that a sequence
$a=(a_k)_{k=0}^{n}$ or $a=(a_k)_{k=0}^{\infty}$ of nonnegative
real numbers is log-concave if
\[
 a_k^2\ge a_{k-1}a_{k+1}
\]
for all relevant indices $k$. For a finite sequence without
internal zeros, log-concavity implies unimodality. Log-concavity
is a classical topic in enumerative combinatorics, closely
related to real-rootedness and Newton's inequalities;
see the surveys of Stanley \cite{Stanley} and
Br\"and\'en \cite{BrandenSurvey}.

To study higher-order log-concavity, define the operator
$\LL$ by
\[
 (\LL a)_k=a_k^2-a_{k-1}a_{k+1}.
\]
Unless essential iteration is explicitly specified, we extend a sequence
by zero outside its original support.
A nonnegative sequence $a$ is said to be $r$-log-concave
if $\LL^j a$ is nonnegative for $1\le j\le r$.
It is called infinitely log-concave if $\LL^j a$ is
nonnegative for every $j\ge0$. Thus infinite log-concavity
requires that nonnegativity be preserved under arbitrarily
many applications of $\LL$. On a fixed support, it is
\emph{strictly infinitely log-concave} if every iterate is
positive at each index of the original positive support.
The notion of infinite log-concavity was introduced by
Boros and Moll \cite[p.~157]{BorosMoll} in their study of certain
polynomials arising from the evaluation of a quartic integral. They conjectured that the coefficient
sequences of these polynomials are infinitely log-concave.
This conjecture stimulated further investigations of
higher-order log-concavity and led to several related
problems for combinatorial sequences
\cite{UminskyYeats,McNamaraSagan,Branden,BrandenChasse}.

An important result in this direction is due to
Br\"and\'en \cite{Branden}. If a polynomial has nonnegative
coefficients and only real nonpositive zeros, then the
generating polynomial of its coefficient sequence after
applying $\LL$ has the same properties. Consequently,
the original coefficient sequence is infinitely log-concave.
Br\"and\'en thereby settled a conjecture proposed independently by Fisk,
McNamara and Sagan \cite[Conjecture~7.1]{McNamaraSagan}, and Stanley.
Another approach relies on stronger log-concavity inequalities
that are preserved under repeated applications of $\LL$
\cite{CravenCsordas,UminskyYeats,McNamaraSagan}.
Neither of these two criteria applies directly to all five
families considered here.

Hou and Zhang \cite{HouZhang} proved eventual $r$-log-concavity
of the partition function for every fixed positive integer $r$.
Hou and Li \cite[Theorem~3.2]{HouLi} gave more general criteria
based on asymptotic ratio expansions. In these fixed-order results,
the starting index may depend on the order. These estimates alone
do not establish positivity of all iterates on a common tail or
resolve the boundary conditions of the original sequence.
We study the logarithmic curvature $h_j$ of the $j$th iterate through
\[
 h_{j+1}=2h_j-\D\log(1-e^{-h_j}),
\]
where $\D$ is the centered second-difference operator. Each iteration
doubles the curvature up to this correction. By holomorphic interpolation
on nested disks, we estimate the correction through a number of
iterations that may grow with the sequence index. We combine these
estimates with exact rational bounds near the endpoints to prove
positivity under repeated applications of $\LL$. The precise local
criterion and boundary conventions are given in Section~\ref{sec:stopping}.

Our approach was developed independently of Liu and Mao's work
on the Taylor coefficients of the Riemann xi-function. Liu and Mao
\cite{LiuMao} also use the curvature recurrence, amplification on
shrinking disks, and a closure argument with a stopping time at each
position. They obtain estimates valid at growing iteration depths.
Here we treat curvature size and analytic radius separately, extend
the estimates to further curvature scales, and account for the
boundary conventions of the applications below. The precise
comparison is given in Remark~\ref{rem:liumao-comparison}.

We apply the criterion first to the two Boros--Moll families.
For the coefficient sequences, we give a new proof of strict
infinite log-concavity and an explicit iteration bound in
Theorem~\ref{thm:rows}, complementing the real-rootedness proof of
Xie and Zhang \cite{XieZhangRows}. We write each coefficient as a
binomial factor times an integral correction. For the transposed
sequences $(d_\ell(\ell+k))_{k\ge0}$, we settle the conjecture of
Zhao \cite{ZhaoPublished} for every integer $\ell\ge3$
in Theorem~\ref{thm:main}.
After finitely many preliminary iterations, we apply holomorphic
estimates at increasing depths. Earlier work on these two families is
discussed in Sections~\ref{sec:rows} and~\ref{sec:tails}.

The other three applications concern Euler's difference table,
the Medina--Moll--Rowland polynomials, and power sequences.
Chen, Gu, Ma, and Wang \cite{EulerOriginal} proved second-order interior
inequalities for Euler's table and posed an infinite-order question.
For the normalized sequences in Euler's table, we discard both
endpoints after each application of $\LL$. In Theorem~\ref{thm:euler},
we prove that every remaining entry is positive for $n\ge3$.
We also confirm the conjecture of Medina, Moll, and Rowland
\cite[Conjecture~4.3 and Theorem~4.5]{MMROriginal} on the coefficient
sequences of polynomials arising from iterated primitives of
$\log(1+x)$. Their result for $n\le300$ is extended to every $n\ge1$, with an explicit iteration bound, in Theorem~\ref{thm:mmr}.
Finally, we answer the question of Br\"and\'en and Chasse
\cite[Question~1]{BrandenChasse} in Theorem~\ref{thm:powers}:
for positive integers $d$,
$(k^d)_{k\ge0}$ is infinitely log-concave if and only if $d\ne2$.
For every real $d\ge3$, all its iterates are strictly positive at
indices $k\ge1$. Bounds on the required number of iterations are stated in
Corollary~\ref{cor:power-entry-bounds}.

Table~\ref{tab:five-applications} compares initial logarithmic
curvatures and iteration bounds, with $h_0$, $T_q$, and $q_*$
as defined in Section~\ref{sec:stopping}.
For the Medina--Moll--Rowland polynomials, $H_n=\sum_{s=1}^n1/s$ and the entries shown are
the reversed coefficients of $n!A_n$. For the finite sequences, these
asymptotic expansions hold uniformly in $n$, as $b\to\infty$ for the
Boros--Moll and Medina--Moll--Rowland sequences and as $k\to\infty$ for Euler.
The other two expansions have fixed $\ell$ or $d$.

\begin{table}[htbp]
\centering
\small
\renewcommand{\arraystretch}{1.45}
\setlength{\tabcolsep}{5pt}
\begin{tabular}{@{}
 >{\raggedright\arraybackslash}p{.23\textwidth}
 >{\raggedright\arraybackslash}p{.25\textwidth}
 >{\raggedright\arraybackslash}p{.30\textwidth}
 >{\raggedright\arraybackslash}p{.15\textwidth}@{}}
\toprule
Sequence and notation & Initial curvature $h_0(k)$
& Iteration bound & Range\\
\midrule
Boros--Moll\par
$a_k=d_k(n)$\par
$0\le k\le n$
& $\displaystyle\frac1k+\frac1{n-k}+O(b^{-2})$
& $T_{1/3}(k)\le\lfloor\log_2n\rfloor+1$
& $n\ge2$\\
\addlinespace[6pt]
Transposed Boros--Moll\par
$a_k=d_\ell(\ell+k)$\par
$k\ge0$
& $\displaystyle\frac{\ell-1/4}{k^2}+O_\ell(k^{-5/2})$
& $T_{q_*}(k)=2\log_2k+O_\ell(1)$\par
as $k\to\infty$
& $\ell\ge3$\\
\addlinespace[6pt]
Euler\par
$\widehat E_{n,k}=E_{n,k}/k!$\par
$0\le k\le n$
& $\displaystyle\frac1k+O(k^{-2})$
& $T_{1/3}(k)\le J(k)$\par
$J(k)=\lceil\log_2k\rceil+1$
& $n\ge750$\par
$k\ge128$\par
$J(k)<n-k$\\
\addlinespace[6pt]
Medina--Moll--Rowland\par
$a_k=\binom nk(H_n-H_k)$\par
$0\le k\le n-1$
& $\displaystyle\frac1k+\frac1{n-1-k}+O(b^{-2})$
& $T_{3/8}(k)\le$\par
$\max\{8,\lfloor\log_2n\rfloor+1\}$
& $n\ge1$\\
\addlinespace[6pt]
Powers\par
$a_k=k^d$, $k\ge0$
& $\displaystyle\frac d{k^2}+O_d(k^{-4})$
& $T_{3/8}(k)\le$\par
$\displaystyle\max\!\left\{0,\left\lceil\log_2\frac{3k^2}{d-2}\right\rceil\right\}$
& $d\ge3$\par
$k\ge256$\\
\bottomrule
\end{tabular}

\medskip
\begin{minipage}{\textwidth}
\footnotesize
Here $b=\min\{k,n-k\}$ for the Boros--Moll sequences and
$b=\min\{k,n-1-k\}$ for the Medina--Moll--Rowland sequences.
\end{minipage}
\caption{Initial curvatures and iteration bounds for the five families.}
\label{tab:five-applications}
\end{table}

The iteration count depends on the initial curvature. If its
holomorphic asymptotic in a sector is $\alpha k^{-p}$, with $\alpha>0$
and $0<p<2$, then $T_{q_*}(k)=p\log_2k+O(1)$. For inverse-square
decay, the leading coefficient also changes under iteration. For fixed
$\alpha>2$, we obtain $T_{q_*}(k)=2\log_2k-\log_2(\alpha-2)+O(1)$,
with error depending on the initial function and $\alpha$.
The leading-coefficient recurrence and the critical value $2$ already
occur in the fixed-order analysis of Hou and Li
\cite[Theorem~3.2, equation~(3.10)]{HouLi}.
Gajdzica \cite[Theorems~4.6--4.7 and Corollary~4.11]{Gajdzica}
treats polynomial-type expansions, and Banerjee, Bringmann, and Rolen
\cite[Theorem~1.13 and Section~5]{BanerjeeBringmannRolen} formulate
general fixed-iterate asymptotics.

The remainder of this paper is organized as follows.
In Section~\ref{sec:stopping}, we establish a local criterion for
infinite log-concavity and derive the analytic and rational estimates
used in its applications. In Sections~\ref{sec:rows} and~\ref{sec:tails},
we study the Boros--Moll coefficient sequences and their transposed
counterparts, respectively. Section~\ref{sec:euler} concerns the
normalized sequences in Euler's difference table under essential
iteration. The Medina--Moll--Rowland polynomials are treated in
Section~\ref{sec:mmr}, followed by power sequences in
Section~\ref{sec:powers}. Finally, in Section~\ref{sec:consequences}
we derive asymptotic estimates for $T_{q_*}(k)$ and investigate
the limiting behavior of the iterates.
\FloatBarrier

\section{A sufficient condition for infinite log-concavity}\label{sec:stopping}

In this section, we give a sufficient condition for positivity under
repeated applications of the log-concavity operator, for both fixed
support and essential iteration. We use local ratio bounds at depths
that may depend on the index, and establish complex-analytic and
exact rational estimates for the applications.

\subsection{A local criterion for iterated positivity}
In the five applications, $k$ denotes the sequence index and $j$ the
iteration depth. Once the family parameter is fixed, we write
$a_k^{[j]}=(\LL^j a)_k$.
For fixed support, let $\mathcal I$ be a finite interval of integers or a
half-line bounded below, and let $a$ be positive on $\mathcal I$ and zero
outside it.  Integer translation only relabels the entries, so we may
normalize $\mathcal I$ to $\{0,\ldots,n\}$ or $\mathbb Z_{\ge0}$ when convenient.
In particular, the positive support of $(k^d)_{k\ge0}$ is
$\{1,2,\ldots\}$. Its zero at $k=0$ lies outside this support.
For essential iteration, take $\mathcal I=\{0,\ldots,n\}$ and retain
$\{j,\ldots,n-j\}$ at layer $j$.
Put $q_*=(3-\sqrt5)/2$ and $\rho_*=-\log q_*$.
At an interior index, whenever the neighboring terms are positive, write
\[
 c_j(k)=\frac{a_{k-1}^{[j]}a_{k+1}^{[j]}}{(a_k^{[j]})^2},\quad
 s_j(k)=1-c_j(k),\quad h_j(k)=-\log c_j(k).
\]
At an endpoint of the original support, where the sequence is extended
by zero, set $c_j=0$ and $s_j=1$. Thus
$a_k^{[j+1]}=(a_k^{[j]})^2s_j(k)$, and positivity of the next iterate
is equivalent to $s_j>0$. Write
$\D f(z)=f(z-1)-2f(z)+f(z+1)$ for the centered second-difference
operator. Direct substitution gives
\begin{align}
 c_{j+1}(k)&=c_j(k)^2
 \frac{(1-c_j(k-1))(1-c_j(k+1))}{(1-c_j(k))^2},
 \label{eq:ratio-iteration}\\
 s_{j+1}(k)&=1-\frac{(1-s_j(k))^2s_j(k-1)s_j(k+1)}{s_j(k)^2},
 \label{eq:s-iteration}\\
 h_{j+1}&=2h_j-\D\log(1-e^{-h_j}).\label{eq:curvature-iteration}
\end{align}
Put $g(w)=\Log(1-e^{-w})$, where $\Log$ denotes the principal logarithm.
If $\Rea w>0$, then $\Rea(1-e^{-w})>0$, so $g$ is holomorphic and
$|\Im g|<\pi/2$. The last recurrence applies wherever both unit shifts
are interior. In the analytic recurrence, $\log(1-e^{-h_j})$ is
interpreted as $g(h_j)$. We set $\phi(w)=\log((1-e^{-w})/w)$, extending
the quotient by $1$ at zero and choosing the logarithm near zero
so that $\phi(0)=0$.

For an interpolating function, we use the same notation
\begin{equation}\label{eq:function-turan}
 (\LL f)(z)=f(z)^2-f(z-1)f(z+1),
\end{equation}
using the same function at the shifted arguments. If a sequence agrees with $f$ on an
integer interval, the function and sequence iterates agree whenever
all required initial indices remain in that interval. Zeros imposed
at a sampling boundary need not agree with the interpolating function.

For $R>1$, \emph{$R$-factor log-concavity} means
$a_k^2\ge R a_{k-1}a_{k+1}$ at every interior index, or
$c_j(k)\le1/R$ at layer $j$. This property is preserved by $\LL$
for $R\ge1/q_*=(3+\sqrt5)/2$ \cite[Lemma~2.1]{McNamaraSagan}.
Thus $[0,q_*]$ is preserved by the ratio recurrence under these
positivity assumptions.
More locally, if all relevant $c_j$ lie in $[0,1)$ and $0<q\le q_*$, then
\begin{equation}\label{eq:threshold-preservation}
 c_j(k)\le q\quad\Longrightarrow\quad
 c_{j+1}(k)\le\frac{c_j(k)^2}{(1-c_j(k))^2}\le q.
\end{equation}
For $0<q\le q_*$, let $T_q(k)$ be the first admissible layer $j$
with $c_j(k)\le q$, and set $T_q(k)=\infty$ if no such layer exists.
These depths are counted from the original sequence, including all
preliminary iterations.
We use $q=1/3$ for the Boros--Moll coefficient sequences
$(d_k(n))_{k=0}^n$ and normalized Euler sequences, $q=3/8$ for the coefficient sequences of the Medina--Moll--Rowland polynomials and for powers, and $q=q_*$ for the iteration estimates for transposed sequences.
Reaching this interval is sufficient, but not necessary, for infinite
log-concavity. Z.~Liu \cite[Theorem~1.1]{LiuNontermination} gives
finitely supported sequences with strictly positive iterates whose
interior ratios satisfy $c_j(k)=1/2>q_*$ at every layer.

A corresponding local sufficient condition for positive half-line
sequences appears in \cite[Lemma~5.1]{LiuMao}. To include
essential iteration, let $M(k)$ be the last layer containing the
coefficient at $k$. Thus $M(k)=\infty$ for fixed support and
$M(k)=\min\{k,n-k\}$ for essential iteration.

\begin{lemma}\label{lem:local-stop}
Let $a$ be positive on $\mathcal I$, and fix $0<q\le q_*$. Suppose that for
each $k\in \mathcal I$ there is an iteration index
$T(k)\in\mathbb Z_{\ge0}\cup\{\infty\}$, with $T(k)\le M(k)$, such that
\[
 a_k^{[j]}>0\quad(0\le j\le T(k)),\qquad
 c_{T(k)}(k)\le q\quad\text{if }T(k)<M(k).
\]
Then all iterates are positive on their prescribed supports. At each fixed position $k$, any bound $c_T(k)\le q$ persists at
every later layer at which the ratio at $k$ remains admissible.
\end{lemma}

\begin{proof}
Suppose that $r$ is the first layer containing a nonpositive coefficient,
say $a_k^{[r]}$. At fixed endpoints the entries are repeatedly squared, so $k$ is
not such an endpoint. By hypothesis, $T(k)<r\le M(k)$ and
$c_{T(k)}(k)\le q$. Since all iterates before layer $r$ are
positive, the ratios needed in \eqref{eq:threshold-preservation}
belong to $[0,1)$ through layer $r-2$. They are admissible also under
essential iteration, because $k$ is still present at layer $r$.
Consequently $c_{r-1}(k)\le q<1$, and
\[
 a_k^{[r]}=(a_k^{[r-1]})^2(1-c_{r-1}(k))>0,
\]
a contradiction. Once positivity is known at every layer, each
attained bound persists by \eqref{eq:threshold-preservation}.
\end{proof}

For fixed support, we usually take $T(k)<\infty$ at each interior
index. The endpoints satisfy the hypothesis with $T(k)=0$.
For essential iteration, $T(k)=M(k)$ means that positivity is checked
through the last coefficient layer, where no ratio is required.
For fixed support, $T(k)=\infty$ is allowed when positivity at that
position is already known at every layer.

The hypotheses can be checked locally. For fixed support, a sufficient
set of inequalities at a center $k$ is, for some finite $T$,
\begin{equation}\label{eq:local-dependency-cone}
 a_v^{[j]}>0\quad
 (0\le j\le T,\ v\in \mathcal I,\ |v-k|\le T+1-j),
 \qquad c_T(k)\le q.
\end{equation}
Alternatively, one may work with the $s$-recurrence. For
$I\subseteq \mathcal I$, write $I^+$ for its one-step neighborhood in $\mathcal I$.
For iterations through depth $T$, suppose that values
$0<\sigma_j(k)\le1$ are defined on sets $I_0,\ldots,I_T$ satisfying
$I_{j+1}^+\subseteq I_j$, agree with the initial
$s$-values at layer zero, and satisfy \eqref{eq:s-iteration}
thereafter. Only original zero-extension endpoints are assigned
$\sigma=1$. By induction, these values agree with $s_j$ on the
retained sets, and the coefficients there are positive. If $k\in I_T$
and $\sigma_T(k)\ge1-q$, the hypothesis of the lemma follows.
Under essential iteration only the available layers are retained.
A newly exposed edge is not a zero-extension boundary.

\subsection{Holomorphic estimates}\label{subsec:prepared-entry}
Liu and Mao \cite[Lemma~3.4]{LiuMao} also estimate the iterated curvature
on shrinking disks. We use the bounded imaginary part of
$\Log(1-e^{-h})$ to estimate the correction independently of the
size of the curvature. The lower bound and disk radius can therefore
be chosen separately, without an upper bound for $|h_0|$ on that disk.
We write $\norm f R$ for the supremum norm on the relevant disk
of radius $R$.

\begin{lemma}\label{lem:strip-difference}
If $F$ is holomorphic and bounded on a disk, then on a concentric
disk whose radius is smaller by an amount $D>1$,
\begin{equation}\label{eq:modulus-difference}
 \norm{\D F}{\mathrm{inner}}\le\frac{2\norm F{\mathrm{outer}}}{D^2-1}.
\end{equation}
If $g$ is holomorphic on a disk of radius $R$ and $|\Im g|\le\pi/2$,
then, on the concentric disk of radius $0\le r<R-1$,
\begin{equation}\label{eq:strip-difference}
 \norm{\D g}r\le
 \frac{4R}{(R+r)((R-r)^2-1)}
 \le\frac4{(R-r)^2-1}.
\end{equation}
If the domain is $\Rea z>-a$, then, for $x=\Rea z+a>1$,
\begin{equation}\label{eq:halfplane-strip}
 |\D g(z)|\le\frac2{x^2-1}.
\end{equation}
\end{lemma}

\begin{proof}
At a point of the inner disk, expand $F(z+w)=\sum_{k\ge0}a_kw^k$.
By Cauchy's estimate and a limit from smaller radii,
$|a_k|\le\norm F{\mathrm{outer}}D^{-k}$. Summing
$\D F(z)=2\sum_{k\ge1}a_{2k}$ proves \eqref{eq:modulus-difference}.

For $g(z+w)=\sum_{k\ge0}a_kw^k$ on $|w|<D$, comparison of Fourier
coefficients gives
\[
 a_kr^k=\frac{i}{\pi}\int_0^{2\pi}
 \Im g(z+re^{it})e^{-ikt}\,dt,\qquad 0<r<D,\ k\ge1.
\]
Rotating the integrand and using
$\int_0^{2\pi}|\cos(kt+\alpha)|\,dt=4$ yields $|a_k|r^k\le2$.
To use the whole outer disk, translate its center to zero and fix $z$
with $t=|z|\le r$. Put $A=R^2-t^2$. The biholomorphic map from the unit disk to the outer disk
\[
 \psi(\zeta)=\frac{z+R\zeta}{1+\overline z\zeta/R}
\]
sends $0$ to $z$, and the preimages of $z+1,z-1$ are
$\xi_+=R/(A-\overline z)$ and $\xi_-=-R/(A+\overline z)$.
Write $g(\psi(\zeta))=\sum_{m\ge0}b_m\zeta^m$; the coefficient
bound just proved gives $|b_m|\le2$ for $m\ge1$.
Expanding the two denominators together and retaining the even or odd
powers of $\overline z$ gives
\[
 |\xi_+^m+\xi_-^m|
 \le\left(\frac R{A-t}\right)^m
       +(-1)^m\left(\frac R{A+t}\right)^m.
\]
Since $t<R-1$, both geometric series converge. Hence
\[
 |\D g(z)|\le
 2\left(\frac R{A-t-R}-\frac R{A+t+R}\right)
 =\frac{4R}{(R+t)((R-t)^2-1)}.
\]
The denominator decreases with $t$, since its derivative is
$-(R-t)(R+3t)-1<0$. This proves \eqref{eq:strip-difference}.
For the half-plane, apply the same coefficient bound to
$H(\zeta)=g(z+2x\zeta/(1-\zeta))$ on the unit disk.
With $u=(2x+1)^{-1}$ and $v=(2x-1)^{-1}$, the linear shifts correspond
to $u$ and $-v$, so
\[
 |\D g(z)|\le2\sum_{k\ge1}(v^k+(-u)^k)
 =2\left(\frac v{1-v}-\frac u{1+u}\right)=\frac2{x^2-1}.
\]
\end{proof}
Consider concentric disks $D_j=\{|z-M|\le R_j\}$, where
$R_0>\cdots>R_K\ge0$ have successive gaps greater than one. Define
\[
 \mathcal E_0=0,\qquad
 \mathcal E_j=\sum_{i=0}^{j-1}
 \frac{2^{1-i}}{(R_i-R_j)^2-1}\quad(1\le j\le K).
\]

\begin{theorem}
\label{thm:unified-entry}
Let $h_0$ be holomorphic near $D_0$, real on its real diameter, and satisfy
\begin{equation}\label{eq:entry-margin}
 \inf_{z\in D_j}\Rea h_0(z)>\mathcal E_j
 \qquad(0\le j\le K).
\end{equation}
Then the curvatures defined by the recurrence are holomorphic on these
disks, have positive real part there, and satisfy
\begin{equation}\label{eq:disk-budget}
 |2^{-j}h_j(z)-h_0(z)|\le\mathcal E_j\qquad(z\in D_j).
\end{equation}
\end{theorem}

\begin{proof}
We use the bound on the imaginary part to obtain an estimate uniform
in the iteration depth. Whenever the curvature has positive real part,
$g(h_i)=\Log(1-e^{-h_i})$ is holomorphic and its imaginary part lies in
$(-\pi/2,\pi/2)$. Telescoping the recurrence gives
\[
 2^{-j}h_j-h_0=-\sum_{i<j}2^{-i-1}\D g(h_i).
\]
On $D_j$, we apply Lemma~\ref{lem:strip-difference} with outer disk
$D_i$ to bound the $i$th summand by $2^{1-i}/((R_i-R_j)^2-1)$.
The estimate \eqref{eq:disk-budget} follows by induction, and all
real parts are positive by \eqref{eq:entry-margin}.
\end{proof}

Using the first bound in \eqref{eq:strip-difference} in the same proof,
we may replace $\mathcal E_j$ in the theorem by
\begin{equation}\label{eq:mobius-error}
 \widetilde{\mathcal E}_0=0,\qquad
 \widetilde{\mathcal E}_j=
 \sum_{i<j}\frac{2^{1-i}R_i}
 {(R_i+R_j)((R_i-R_j)^2-1)}.
\end{equation}
This also applies when $R_K=0$: the strict inequalities ensure that
$h_K$ is holomorphic near the center.

Evaluation of \eqref{eq:disk-budget} at the center gives the bounds used
in the applications. If $A_-\le h_0(M)\le A_+$, then
\begin{equation}\label{eq:entry-center-bounds}
 2^j(A_--\mathcal E_j)\le h_j(M)
 \le2^j(A_++\mathcal E_j).
\end{equation}
In particular, for $0<q\le q_*$, the bound
\begin{equation}\label{eq:entry-terminal-test}
 2^K(A_--\mathcal E_K)\ge\log(1/q)
\end{equation}
implies $e^{-h_K(M)}\le q$.
A direct evaluation of the telescoping sum gives the sharper estimate
\begin{equation}\label{eq:center-error}
 |2^{-j}h_j(M)-h_0(M)|\le\mathcal C_j,
 \qquad \mathcal C_j=\sum_{i<j}\frac{2^{1-i}}{R_i^2-1}.
\end{equation}
Here only $\Rea h_i>0$ on $D_i$ for $i<j$ is required.
If $R_K=0$, then $\mathcal C_K=\widetilde{\mathcal E}_K$.

At an integer center, both unit shifts remain in the preceding disk
because successive radii differ by more than one. If the initial function interpolates
the curvature of positive initial entries, uniqueness of the recurrence
identifies these functions with the curvatures of positive sequence
iterates on the required index ranges. For essential iteration, this identification holds only
while the corresponding entries remain admissible.

\begin{corollary}\label{cor:reference-entry}
In Theorem~\ref{thm:unified-entry}, write $h_0=H+u$, with $H$ and $u$
holomorphic and real on the real diameter.  It suffices to give numbers
$b_j,e_j$ such that
\begin{equation}\label{eq:reference-entry-input}
 \Rea H\ge b_j,\qquad |u|\le e_j\quad\hbox{on }D_j,
 \qquad b_j-e_j>\mathcal E_j.
\end{equation}
If $|u(M)|\le e_c$, one may take $A_\pm=H(M)\pm e_c$ in
\eqref{eq:entry-center-bounds}--\eqref{eq:entry-terminal-test}.
A sharper real bound for $h_0(M)$ may be used instead.
\end{corollary}
\begin{proof}
Use $\Rea(H+u)\ge\Rea H-|u|$ and
$|h_0(M)-H(M)|\le e_c$, then apply the theorem.
\end{proof}

\begin{corollary}\label{cor:geometric-budget}
Let $R_j=R(3/4)^j$. If the smallest successive gap has square at
least $D^2>1$, then
\begin{equation}\label{eq:geometric-budget}
 \mathcal E_j<\frac{42}{(1-D^{-2})R^2},\qquad
 \widetilde{\mathcal E}_j<\frac{26}{(1-D^{-2})R^2}
 \quad(1\le j\le K).
\end{equation}
\end{corollary}
\begin{proof}
Set $b=8/9$. By the definition of $\mathcal E_j$,
\[
 \mathcal E_j\le\frac{2}{(1-D^{-2})R^2}
 \sum_{r=1}^j\frac{b^{j-r}}{(1-(3/4)^r)^2}.
\]
The reciprocals of the denominators are at most $16,6,3,7/3$ for
$r=1,2,3,r\ge4$, respectively.  Since
$16b^2+6b+3<21$ and $(7/3)/(1-b)=21$, the sum is less than $21$;
the cases $j=1,2$ obey the same bound.

For the second estimate, the corresponding sum is
\[
 S_j=\sum_{m=1}^j
 \frac{b^{j-m}}{(1+(3/4)^m)(1-(3/4)^m)^2}.
\]
For the first four terms, the reciprocals of the denominators are at most
$64/7,7/2,13/6,5/3$, respectively,
which give $S_j<13$ for $1\le j\le4$.
For $j\ge5$ the new term is at most
$1073741824/772820587<13/9$.
The recurrence $S_j=(8/9)S_{j-1}
+((1+(3/4)^j)(1-(3/4)^j)^2)^{-1}$ therefore gives $S_j<13$
for every $j$. This proves the constant $26$.
\end{proof}
For geometric radii through $R_{K-1}$, if $R_{K-1}^2>\eta>1$, then
\eqref{eq:center-error} gives
\begin{equation}\label{eq:geometric-center-error}
 \mathcal C_j<\frac{18}{(1-\eta^{-1})R^2}
 \qquad(1\le j\le K),
\end{equation}
since $\sum_{i\ge0}2(8/9)^i=18$. This estimate also permits $R_K=0$.

\begin{corollary}\label{thm:analytic-amplification}
Let $h_0$ be holomorphic near $|z-M|\le R$ and real on its real diameter.
Suppose that
\begin{equation}\label{eq:halfplane-stop}
 \Rea h_0\ge\mu>0,\qquad \mu R^2\ge128.
\end{equation}
Put
\begin{equation}\label{eq:analytic-depth}
 K=\begin{cases}
 0,&\mu\ge\rho_*,\\
 \left\lceil\log_2(2\rho_*/\mu)\right\rceil,&0<\mu<\rho_*,
 \end{cases}
 \qquad R_j=R(3/4)^j\quad(0\le j\le K).
\end{equation}
The iterates $h_j$ are holomorphic with $\Rea h_j>0$ on
$|z-M|\le R_j$ for $0\le j\le K$, and $h_K(M)\ge\rho_*$.
For $0\le j\le K$, we also have
\begin{equation}\label{eq:center-doubling}
 h_0(M)-\frac{54}{R^2}
 <2^{-j}h_j(M)<h_0(M)+\frac{54}{R^2},
\end{equation}
and $\tau=\min\{0\le j\le K:h_j(M)\ge\rho_*\}$ satisfies
\begin{equation}\label{eq:disk-time-sandwich}
\max\!\left\{0,
\left\lceil\log_2\frac{\rho_*}{h_0(M)+54/R^2}\right\rceil\right\}
\le\tau\le K.
\end{equation}
\end{corollary}

\begin{proof}
If $\mu\ge\rho_*$, then $h_0(M)\ge\rho_*$ and $K=\tau=0$,
so all assertions are immediate. Assume $0<\mu<\rho_*$. Then $K\ge2$
and $2^{K-1}\mu<2\rho_*<2$.  The hypothesis $\mu R^2\ge128$ makes
the smallest gap in \eqref{eq:analytic-depth} have square greater than
$9/2$. By Corollary~\ref{cor:geometric-budget},
\begin{equation*}
 \mathcal E_j<54/R^2\le(27/64)\mu.
\end{equation*}
Applying Theorem~\ref{thm:unified-entry} with $A_-=\mu$, we obtain
holomorphic iterates with $\Rea h_j>0$ for $0\le j\le K$ and
$h_K(M)>(37/64)2^K\mu\ge(37/32)\rho_*$.
From the bounds at the center, we also obtain
\eqref{eq:center-doubling} and the lower bound in
\eqref{eq:disk-time-sandwich}.  The final radius exceeds one, so the required unit shifts
remain within the domains of the preceding iterates.  No monotonicity of
$h_j(M)$ is required.
\end{proof}

For curvature iterates at an integer center $n$, let $\tau(n)$ denote
the first layer at which $h_j(n)\ge\rho_*$; set $\tau(n)=\infty$ if
no such layer exists.

\begin{remark}[Comparison with Liu and Mao]\label{rem:liumao-comparison}
For inverse-linear initial curvatures, Liu and Mao
\cite[Lemma~3.4]{LiuMao} construct holomorphic iterates with positive
real part through every depth $0\le j\le\lceil\log_2n\rceil$ under
the hypotheses of that lemma. They prove
\[
 |2^{-j}g_j(n)-g_0(n)|<\frac{1500}{n^2}.
\]
Together with the initial bounds in that lemma, this estimate
gives $\min\{j:g_j(n)\ge\rho_*\}=\log_2n+O(1)$.
In their Lemma~2.2, they identify these analytic curvatures with the
sequence iterates. In their Lemma~5.1, they deduce global positivity from
ratio bounds attained at position-dependent depths on a half-line.
They thus obtain estimates at growing depths and bounds on the number
of iterations needed for inverse-linear initial curvature.

In Corollary~\ref{thm:analytic-amplification}, the curvature lower
bound $\mu$ and the analytic radius $R$ are separate parameters.
The hypotheses require a lower bound on $\Rea h_0$ throughout the
disk but no upper modulus bound there. An upper bound at the center
is still used for two-sided iteration bounds. This also applies to
power-law decay with $0<p<2$ and, after finitely many iterations, to
inverse-square curvature with leading coefficient greater than $2$,
as shown in Theorem~\ref{thm:power-law}.
The global criterion accommodates both fixed support and essential
iteration, with the boundary conventions used in the applications.
\end{remark}

\begin{corollary}\label{cor:normalized-stop}
Let \(x>0\), \(\kappa\in\mathbb R\), and \(0<r\le1/2\).  Suppose that
\(h_0\) is holomorphic on
\(|z-(\kappa+x)|\le rx\), is real on the real diameter, and satisfies
\begin{equation}\label{eq:normalized-disk-bound}
|(z-\kappa)^2h_0(z)-B|\le E
\end{equation}
there, where \(B>0\).  If
\begin{equation}\label{eq:normalized-size}
r^2\left(\frac{B}{(1+r)^2}-\frac{E}{(1-r)^2}\right)\ge128,
\end{equation}
then, for some finite $T\ge0$, the iterates $h_j$
are holomorphic with $\Rea h_j>0$ on nested disks for $0\le j\le T$,
and $h_T(\kappa+x)\ge\rho_*$.
\end{corollary}

\begin{proof}
Write $z=\kappa+x(1+s)$, so $|s|\le r$.  The elementary inversion estimate
\[
 \Rea(1+s)^{-2}\ge(1+r)^{-2},\qquad
 |(1+s)^{-2}|\le(1-r)^{-2}\qquad(r\le1/2)
\]
follows by mapping $|s|\le r$ under $w=(1+s)^{-1}$ and minimizing
$\Rea w^2$. Hence, by \eqref{eq:normalized-disk-bound},
\[
\Rea h_0(z)\ge x^{-2}\left(\frac{B}{(1+r)^2}
-\frac{E}{(1-r)^2}\right)=:\mu>0.
\]
The disk has radius $R=rx$, and \eqref{eq:normalized-size} is exactly
$\mu R^2\ge128$.  Apply Corollary~\ref{thm:analytic-amplification}.
\end{proof}

Factorial and binomial curvatures are explicit. In several applications,
the remaining holomorphic factor can be estimated through a bound on
its argument, as in the following proposition.
Consider positive entries $a_k=B_kF(k)$ on their natural integer
interval. For $B_k=1/k!$ put $d=k$; for $B_k=\binom Nk$ put
$d=\min(k,N-k)$.

\begin{proposition}\label{prop:reference-factors}
Suppose $F$ has a holomorphic logarithm $L$ near $|z-k|\le d/2$,
with \mbox{$|\Im(L-c)|\le C$} there, where $C$ is fixed and $c$ may depend
on the center. For all sufficiently large $d$, uniformly over the family,
the holomorphic iterates satisfy $\Rea h_j>0$ on nested disks for
$0\le j\le\tau(k)<\infty$, with $h_{\tau(k)}(k)\ge\rho_*$ and
\[
 \tau(k)=\log_2d+O_C(1).
\]
Here $\tau(k)$ is defined by the holomorphic recurrence. It equals
$T_{q_*}(k)$ for the sequence if every integer in
$[k-\tau(k)-1,k+\tau(k)+1]$ belongs to the original positive support.
\end{proposition}
\begin{proof}
Set $L_*(z)=(L(z)+\overline{L(\overline z)})/2$. This function is real on the real diameter and equals
$\log F$ at integers.  If $v=\Im(L-c)$, then
$\Im L_*(z)=(v(z)-v(\overline z))/2$, so $|\Im L_*|\le C$.
It follows from Lemma~\ref{lem:strip-difference} that
$|\D L_*|=O_C(d^{-2})$ on $|z-k|\le d/4$.
The factorial and binomial factors have respective curvatures
\[
 h_{\mathrm{ref}}(z)=\log(1+1/z),\qquad
 h_{\mathrm{ref}}(z)=\log(1+1/z)+\log(1+1/(N-z)).
\]
On this disk one denominator is within $d/4$ of $d$, and the other,
if present, has positive real part.  Hence $\Rea h_{\mathrm{ref}}\ge c_1/d$ and
$h_{\mathrm{ref}}(k)\le2/d$. Thus $h_0=h_{\mathrm{ref}}-\D L_*$
has positive real part of order $d^{-1}$ on a disk of radius
$d/4$, while $h_0(k)=O_C(d^{-1})$. For large $d$, the upper and lower
bounds in Corollary~\ref{thm:analytic-amplification} give
$\tau(k)=\log_2d+O_C(1)$. If the stated initial interval lies in the
positive support, $h_0$ agrees with the original curvatures at all
required integer arguments. The recurrences then agree through
layer $\tau(k)$, giving $T_{q_*}(k)=\tau(k)$.
\end{proof}

Under essential iteration, the support condition requires
$\tau(k)<M(k)$, so that the ratio at layer $\tau(k)$ is defined.

For Boros--Moll sequences, the logarithm of the Laplace correction is a
difference of two logarithms with bounded argument. For the Medina--Moll--Rowland polynomials,
the correction is a Stieltjes sum with positive real part.  In Euler's table,
we factor $n!e^{z/n-1}$ out of the interpolating polynomial and
bound the logarithm of the remaining factor.
Positive geometric factors leave the ratios unchanged. If $b_k=AB^ka_k$
with $A,B>0$, then
\[
 (\LL^j b)_k=A^{2^j}B^{2^jk}(\LL^j a)_k.
\]

If the hypotheses of the holomorphic estimates fail at layer zero,
we may check them
after finitely many positive iterations. To estimate these iterates,
we express the curvature as a reference term and a remainder.

\begin{lemma}\label{lem:reference-scale}
Let $H$ be independent of $j$ and admit a holomorphic logarithm
on a domain containing all required shifted arguments. Write
\[
 h_j(z)=2^jH(z)\{\beta_j+R_j(z)\}.
\]
Suppose the logarithms admit holomorphic branches on a common connected
domain containing the required shifted arguments. The curvature recurrence
is then equivalent to
\begin{equation}\label{eq:reference-scale}
\begin{aligned}
 R_{j+1}={}&R_j+\beta_j-\beta_{j+1}
 -\frac{\D\log H}{2^{j+1}H}\\
 &-\frac1{2^{j+1}H}\D\!\left\{
 \log(\beta_j+R_j)
 +\phi\bigl(2^jH(\beta_j+R_j)\bigr)\right\}.
\end{aligned}
\end{equation}
\end{lemma}

\begin{proof}
The relevant logarithmic identities hold up to additive constants,
which are annihilated by $\D$. We may therefore use
$\log(1-e^{-h_j})=\log h_j+\phi(h_j)$ in
\eqref{eq:curvature-iteration}. The constant $j\log2$ also disappears
under $\D$, and division by $2^{j+1}H$ gives \eqref{eq:reference-scale}.
\end{proof}

The holomorphic estimates also apply from layer $J$. Suppose
$a_v^{[j]}>0$ for $0\le j\le J$ on all required index ranges, and
$h_J$ agrees with a holomorphic function $h$ at the integer points
of its initial disk. Starting from $h$ at layer $J$, define the subsequent
analytic curvatures by \eqref{eq:curvature-iteration}. If the estimates
apply through $K$ further layers and give
$h_{J+K}(k)\ge\log(1/q)$, where $0<q\le q_*$, uniqueness identifies
these functions with the sequence curvatures wherever the latter are
admissible. Thus $a_v^{[j]}>0$ for $0\le j\le\min\{J+K,M(v)\}$
on the required index ranges, and $c_{J+K}(k)\le q$ if $J+K<M(k)$.
These index ranges must contain all neighboring coefficients used at
each stage and all earlier coefficients needed to compute them.

\subsection{Exact bounds for finite parameter ranges}\label{sec:models}
For bounded parameter intervals, we approximate the entries and ratios
by polynomials with explicit error bounds.
We use monomials on $[0,1]$ or $[-1,1]$ and let $\|P\|_1$ be the
sum of absolute coefficients. Then $\sup|P|\le\|P\|_1$ and
$\|PQ\|_1\le\|P\|_1\|Q\|_1$.
We say that $(P,E)$ encloses $f$ if $|f-P|\le E$, where $P$ is a
polynomial with rational coefficients and $E$ is a nonnegative rational
number. The same bounds hold for multivariate monomials on products
of these intervals.

\begin{lemma}\label{lem:model-arithmetic}
If $(P,E)$ and $(Q,F)$ enclose $f$ and $g$, then addition and subtraction
have error $E+F$. In the displayed enclosures, any irrational error term
is understood to be replaced by a rational upper bound.
For any polynomial $A$ with rational coefficients, the product is enclosed by
\[
 \left(A,\|PQ-A\|_1+E\|Q\|_1+F\|P\|_1+EF\right).
\]
If $g\ge d>0$, an enclosure for the quotient is
\[
 \left(A,\frac{\|P-QA\|_1+E+F\|A\|_1}{d}\right).
\]
In particular, if $f\ge d>0$, an enclosure for its reciprocal is
$\bigl(A,(\|1-PA\|_1+E\|A\|_1)/d\bigr)$.
\end{lemma}

\begin{proof}
Expand $fg-PQ=(f-P)Q+P(g-Q)+(f-P)(g-Q)$, and use
$f/g-A=(P-QA+(f-P)-(g-Q)A)/g$.  The reciprocal is the case $f=1$
of the quotient rule. The error includes every discarded coefficient,
and the enclosure is valid for any approximating polynomial with rational coefficients satisfying
the stated residual bound.
\end{proof}

At each prescribed scale $S=2^b$, we round retained polynomial
coefficients downward to multiples of $S^{-1}$ and error bounds upward.
For products, we truncate at the stated degree and include all discarded
coefficients and rounding errors in $\|PQ-A\|_1$.
In the Boros--Moll coefficient, Euler, and power computations, we choose
quotient approximations by formal division. If
$P=\sum p_mx^m$ and $Q=\sum q_mx^m$, with $q_0>0$, set
\[
 a_m=\frac1S\left\lfloor\frac{S}{q_0}
 \left(p_m-\sum_{i=1}^m q_i a_{m-i}\right)\right\rfloor
 \qquad(0\le m\le D),
\]
where $D$ is the stated degree and missing coefficients are zero.
The quotient error is computed from the full residual $P-QA$ for
$A=\sum_{m=0}^D a_mx^m$, including terms of degree greater than $D$.
For the transposed sequences, we first use this rule with $P=1$,
certify the reciprocal, and then apply the product bound. In their
initial Riccati recurrence we instead estimate the whole difference,
as described in Subsection~\ref{sec:certificates}.
The bivariate quotient approximations are specified in the proof of
Proposition~\ref{ver:mmr}.
For each family, neighboring entries use the same parameter, so the
inequalities hold on whole parameter intervals.

Fixed-parameter computations use rational intervals with outward
rounding. Before every division, a
strictly positive lower bound for its denominator is established.
Before every ratio update, we prove that every retained ratio at the
preceding layer has an upper bound below one. Nonnegativity follows inductively
from the initial ratios and the recurrence. The coefficient iterates are therefore positive. Positivity of
$(1-c_j(k))^2$ alone would not imply positivity of the next coefficient.
Once a calculation
certifies $c_T(k)\le q\le q_*$, these inequalities give the hypotheses
of Lemma~\ref{lem:local-stop} at $k$.

The support convention determines the initial data required by each
calculation. A coefficient $a_k^{[j]}$ depends only on the initial
entries with indices in $[k-j,k+j]$, and its ratio $c_j(k)$ depends
on those in $[k-j-1,k+j+1]$. This follows by induction from the
three-term definition of $\LL$. Thus a finite calculation includes
these intervals, with zeros assigned only outside the original fixed
support. Under essential iteration both endpoints are discarded at
each step. An artificial truncation endpoint is also discarded and is never
assigned a zero value. In the following sections, we specify the initial intervals for each
local computation and their overlap with the ranges where the
analytic estimates hold.

\section{Boros--Moll sequences}\label{sec:rows}
The Boros--Moll polynomials arise from the evaluation
\[
 \int_0^\infty\frac{dt}{(t^4+2xt^2+1)^{n+1}}
 =\frac{\pi P_n(x)}{2^{n+3/2}(x+1)^{n+1/2}},
 \qquad x>-1,
\]
where $n$ is a nonnegative integer \cite{BorosMollIntegral}.
Their coefficients have the positive sum representation
\begin{equation}\label{eq:coefficients}
 P_n(x)=\sum_{k=0}^n d_k(n)x^k,\qquad
 d_k(n)=4^{-n}\sum_{r=k}^n2^r
 \binom{2n-2r}{n-r}\binom{n+r}{r}\binom rk.
\end{equation}

Boros and Moll \cite{BorosMoll1999} proved that these coefficient
sequences are unimodal. Kauers and Paule \cite{KauersPaule} proved
their log-concavity using coefficient recurrences, and Chen, Pang, and
Qu \cite{ChenPangQuLogConcavity} gave a combinatorial proof using
partially $2$-colored permutations. Chen and Xia
\cite{ChenXiaRatio,ChenXia2LC} established ratio monotonicity and
$2$-log-concavity. Chen, Dou, and Yang \cite{ChenDouYang2013} proved
Br\"and\'en's conjectures on the real-rootedness of two factorially
normalized coefficient polynomials. These real-rootedness results
imply $2$- and $3$-log-concavity of the Boros--Moll coefficient sequences.

Boros and Moll \cite[p.~157]{BorosMoll} conjectured nonnegativity
at every iteration depth. The polynomials
$P_n(x)$ are not real-rooted for $n\ge2$ \cite{BorosMoll1999}, so
the preservation theorem of Br\"and\'en \cite{Branden} does not apply
to them directly.
Chen, Yang, and Zhang \cite[Conjecture~1.3]{ChenYangZhang2018}
instead conjectured that the generating polynomial of the first
iterate has only real zeros. Xie and Zhang
\cite{XieZhangRows} proved this stronger statement by showing
strict interlacing with a Narayana polynomial. Infinite log-concavity
then follows from Br\"and\'en's theorem.

Here we give a proof through the criterion of
Section~\ref{sec:stopping}, together with an explicit bound on the
number of iterations needed to reach $3$-factor log-concavity.
We express each coefficient as a binomial factor times an integral
and estimate the logarithmic curvature of each factor to prove
the required inequalities in the interior. Near the endpoints,
we verify them using exact rational bounds.
For each $n\ge0$, put $a_k=d_k(n)$ for $0\le k\le n$, extend by zero,
and write $a_k^{[j]}=(\LL^j a)_k$.

\begin{theorem}\label{thm:rows}
For every integer $n\ge0$, put
$a^{[j]}=\LL^j(d_0(n),\ldots,d_n(n))$, with zero extension.
Then $a_k^{[j]}>0$ for every $j\ge0$ and $0\le k\le n$.
If $n\ge2$ and $j\ge\lfloor\log_2n\rfloor+1$, then
\begin{equation}\label{eq:row-main}
 (a_k^{[j]})^2\ge3a_{k-1}^{[j]}a_{k+1}^{[j]},
 \qquad 1\le k<n.
\end{equation}
\end{theorem}

\subsection{Moment representation and initial estimates}\label{sec:integrals}
We begin with the moment representation. The binomial factors in
this section and the next have respective curvatures
\begin{equation}\label{eq:two-profiles}
 \log\left(1+\frac{n+1}{z(n-z)}\right),\qquad
 \log\left(1+\frac{\lambda}{z(z+\lambda+1)}\right).
\end{equation}
The rational functions inside these logarithms have poles on opposite
sides of the index interval in the first case and on the same side in
the second. We estimate
the integral corrections to apply the criterion in Section~\ref{sec:stopping}.

For $n\ge0$ and $0\le r\le n$, put
\[
 A_{n,r}=4^{-n}2^r
 \binom{2n-2r}{n-r}\binom{n+r}{r},
 \qquad d_k(n)=\sum_{r=k}^n A_{n,r}\binom rk.
\]
Set $A_{n,r}=0$ for all other integer $r$.  Fix $n\ge1$.
We use the coefficient recurrence from
Kauers and Paule~\cite{KauersPaule} and Chen, Pang, and
Qu~\cite{ChenPangQuLogConcavity},
\begin{equation}
 (k+1)(k+2)d_{k+2}(n)
 -(2n+1)(k+1)d_{k+1}(n)
 +(n-k)(n+k+1)d_k(n)=0,
 \label{row:eq:coefficient-recurrence}
\end{equation}
for $0\le k<n$, with $d_{n+1}(n)=0$.

\begin{lemma}\label{lem:moment}
For every integer $n\ge0$ and $0\le k\le n$,
\begin{equation}\label{eq:moment}
d_k(n)=\binom nk C_n
\int_0^1x^{-1/2}(1-x)^{n+k}(1+x)^{n-k}\,dx,
\qquad
C_n=\frac{2^{-n}\Gamma(2n+3/2)}{\sqrt\pi\,\Gamma(n+1)^2}.
\end{equation}
\end{lemma}

\begin{proof}
The case $n=0$ is immediate. Assume $n\ge1$ and put
$e_k=d_k(n)/\binom nk$. The coefficient recurrence
\eqref{row:eq:coefficient-recurrence} gives
\[
 (n-k)e_{k+1}=(2n+1)e_k-(n+k)e_{k-1}
 \qquad(1\le k<n),
\]
together with the terminal relation $(2n+1)e_n=2ne_{n-1}$.
We verify by integration by parts that the integrals in
\eqref{eq:moment} satisfy the same recurrence and terminal relation.
At $k=n$, the integral is $\mathrm B(1/2,2n+1)$, so
$d_n(n)=2^{-n}\binom{2n}n$ gives the stated value of $C_n$.
The terminal relation and backward recurrence determine all remaining terms.
\end{proof}

The substitution $x=\tanh(t/2)$ in \eqref{eq:moment} gives
\begin{equation}\label{row:eq:moments}
 d_k(n)=\frac{C_n}{2}\binom nk\mu_n(k),\qquad
 \mu_n(z)=\int_0^\infty e^{-zt}w_n(t)\,dt,
\end{equation}
where
\[
 w_n(t)=\bigl(\sinh(t/2)\bigr)^{-1/2}
         \bigl(\cosh(t/2)\bigr)^{-2n-3/2}.
\]
The estimates $w_n(t)\sim\sqrt{2/t}$ at zero and
$w_n(t)=O(e^{-(n+1)t})$ at infinity imply that $\mu_n$ is holomorphic on
$\Rea z>-n-1$. To choose a holomorphic logarithm, we first show that it
has no zeros on a smaller half-plane.
Set $v_n(t)=e^{\sqrt n\,t}w_n(t)$.
By the arithmetic--geometric mean inequality,
\[
 -\frac{w_n'(t)}{w_n(t)}
 =\frac14\coth(t/2)+\left(n+\frac34\right)\tanh(t/2)
 \ge\sqrt{n+\frac34}>\sqrt n.
\]
Thus $v_n$ is integrable and strictly decreasing.
For $x=\operatorname{Re} z>-\sqrt n$, integration by parts yields
\[
 \begin{aligned}
 F_n(z)&:=(z+\sqrt n)\mu_n(z)
   =\int_0^\infty(1-e^{-(z+\sqrt n)t})(-v_n'(t))\,dt,\\
 \operatorname{Re} F_n(z)&\ge F_n(x)>0.
 \end{aligned}
\]
At the origin, $-v_n'(t)=O(t^{-3/2})$, so the integrand is
$O(t^{-1/2})$ and the boundary term is $O(t^{1/2})$. At infinity,
$v_n$ and $v_n'$ decay exponentially. The integral therefore converges
and both boundary terms vanish, justifying the integration by parts.
Thus $\operatorname{Log}F_n$ is holomorphic on
$\operatorname{Re} z>-\sqrt n$ and $|\operatorname{Im}\operatorname{Log}F_n|<\pi/2$.
On the same half-plane, $\operatorname{Log}F_n(z)-\operatorname{Log}(z+\sqrt n)$
defines a holomorphic logarithm of $\mu_n(z)$.

We derive a lower bound for the initial curvature from the coefficient
recurrence and estimate its holomorphic continuation using the moment
formula. Recall that $s_j(k)=1-c_j(k)$.
For $0\le k<n$, define
\[
 p_k=n-\frac{(k+1)d_{k+1}(n)}{d_k(n)},\qquad p_n=n.
\]
By the coefficient formula,
\[
 (n-k)d_k(n)-(k+1)d_{k+1}(n)
 =\sum_{r=k}^n A_{n,r}\binom rk(n-r)>0.
\]
The term with $r=k$ is positive.  Together with $d_{k+1}(n)>0$, this gives
\begin{equation}
 k<p_k<n.
 \label{row:eq:p-range}
\end{equation}
Substitution in~\eqref{row:eq:coefficient-recurrence} gives the backward
Riccati recurrence
\begin{equation}
 p_k=\frac{np_{k+1}+k(k+1)}{n+1+p_{k+1}},\qquad 0\le k<n.
 \label{row:eq:riccati}
\end{equation}
The equivalent forward recurrence is
$p_{k+1}=((n+1)p_k-k(k+1))/(n-p_k)$, whose denominators are positive.
For an interior index $1\le k<n$, put $d=\min\{k,n-k\}\ge1$.
We also have $n-p_{k-1}=(n+k)(n-k+1)/(n+1+p_k)$. Consequently,
\begin{equation}
 s_0(k)
 =1-\frac{k(n-p_k)(n+1+p_k)}
 {(k+1)(n+k)(n-k+1)}.
 \label{row:eq:s-initial}
\end{equation}
Since $p_k>k$,
\[
\begin{aligned}
 c_0(k)
 &<\frac{k(n-k)(n+k+1)}{(k+1)(n+k)(n-k+1)}\\
 &<\min\left\{\frac{k}{k+1},\frac{n-k}{n-k+1}\right\}
 =\frac{d}{d+1}.
\end{aligned}
\]
The comparisons with $k/(k+1)$ and $(n-k)/(n-k+1)$ follow from
$2k>0$ and $n>0$, respectively.
The first bound is equivalent to the lower bound of Chen and Gu
\cite[Theorem~1.2]{ChenGu}.
It follows that $s_0(k)>1/(d+1)$ and
\begin{equation}
 h_0(k)>\log(1+1/d)\ge\frac1d-\frac1{2d^2}.
 \label{row:eq:initial-bound}
\end{equation}

For $n\ge1024$, fix $k$ with $d=\min(k,n-k)\ge48$.
The moment representation gives the initial curvature
\[
 h_0(z)=\Log(1+1/z)+\Log(1+1/(n-z))
       +\Log\bigl(1-(z+\sqrt n)^{-2}\bigr)-\D\Log F_n(z).
\]
It is holomorphic on $1<\Rea z<n$ and agrees with $-\log c_0$ at
interior integer arguments there. For $|w-d|\le rd$, we have
$\Rea\Log(1+1/w)\ge((1+r)d+1)^{-1}$.
Applying this to the nearer endpoint term, and using
\eqref{eq:halfplane-strip} with $\sqrt n\ge32$, gives,
for $0\le r\le19/20$ and $|z-k|\le rd$,
\begin{equation}\label{eq:bm-mobius-seed}
 \Rea h_0(z)\ge f_r(d):=
 \frac1{(1+r)d+1}-\frac3{((1-r)d+32)^2-1}.
\end{equation}
At the center, \eqref{row:eq:initial-bound} gives
$h_0(k)>d^{-1}-(2d^2)^{-1}$.

\begin{proposition}\label{row:lem:interior}
For $n\ge1024$ and $d=\min(k,n-k)\ge48$, put
$L=\lfloor\log_2d\rfloor+2$. Then
\[
 a_v^{[j]}>0\quad
 (0\le j\le L,\ 0\le v\le n,\ |v-k|\le L+1-j),
 \qquad h_L(k)>\frac{85}{64}>\log3.
\]
\end{proposition}
\begin{proof}
Take $R_j=(19d/20)(3/4)^j$ for $0\le j<L$ and $R_L=0$,
and write $D_j=\{|z-k|\le R_j\}$.
The initial disk lies in $1<\Rea z<n$. The initial curvature has positive real
part there, because at $r=19/20$ the denominator comparison in
\eqref{eq:bm-mobius-seed} reduces to
\[
 (d/20+32)^2-1-3(39d/20+1)
 =\frac{(d-530)^2}{400}+\frac{1271}{4}>0.
\]
For $0\le r\le57/80$, $d^2f_r(d)$ increases with $d\ge48$.
Indeed, with $A=1+r$, $B=1-r>1/4$, the positive part of its derivative
is $A^{-1}(1-(Ad+1)^{-2})>49/100$, and the magnitude of its negative
part is bounded by
\[
 192(1024/1023)^2\frac{d}{(Bd+32)^3}
 \le\frac{(1024/1023)^2}{36B}<\frac18.
\]
Also $f_r(d)$ decreases with $r$. Substitution of $d=48$ at
$r_j=(19/20)(3/4)^j$ gives
\[
 d^2\inf_{D_1}\Rea h_0>24,\quad
 d^2\inf_{D_2}\Rea h_0>28,\quad
 d^2\inf_{D_3}\Rea h_0>31,\quad
 d^2\inf_{D_j}\Rea h_0>34\quad(j\ge4).
\]
Among $R_0,\ldots,R_{L-1}$, all successive gaps have square greater than seven. For
$L=7$ use $d\ge48$; for $L\ge8$ use $d\ge2^{L-2}$, with the smallest
bound at $L=8$. The common lower bound is
$((19\cdot48/80)(3/4)^5)^2>7$.
By Corollary~\ref{cor:geometric-budget},
\[
 d^2\widetilde{\mathcal E}_j<
 \frac{26}{(1-1/7)(19/20)^2}<34\qquad(1\le j<L).
\]
At the first three iterations, we obtain the sharper bounds
\[
 d^2\widetilde{\mathcal E}_1<21,\qquad
 d^2\widetilde{\mathcal E}_2<26,\qquad
 d^2\widetilde{\mathcal E}_3<28.
\]
Each normalized sum decreases with $d$, so these comparisons follow
by substituting $d=48$ in \eqref{eq:mobius-error}.
Thus $\inf_{D_j}\Rea h_0>\widetilde{\mathcal E}_j$ for $j<L$.
The last positive radius is three times the smallest of these gaps;
the displayed lower bound gives $R_{L-1}^2>64$. Using
\eqref{eq:geometric-center-error}, we obtain
\[
 \mathcal C_L<\frac{18}{(1-1/64)(19/20)^2d^2}<\frac{21}{d^2}.
\]
In particular, $h_0(k)>\mathcal C_L$.
Theorem~\ref{thm:unified-entry}, with the errors in
\eqref{eq:mobius-error}, now gives
\[
 h_L(k)>\frac{2^L(d-43/2)}{d^2}\ge\frac{85}{64}.
\]
Indeed, $(d-43/2)/d^2$ decreases for $d\ge48$.
The upper endpoint $d=64$ gives $85/64$ when $L=7$;
for $L\ge8$, $d=2^{L-1}$ gives
$2-43/2^{L-1}\ge213/128>85/64$. For $1\le j\le L$, we have $R_{j-1}>L+1-j$, since
$R_{L-1}>8$ and the successive gaps exceed one.
Also $L+1<R_0+1<d$. Hence all initial indices required for the
stated iterates lie in the original support. Induction using
$a_v^{[j+1]}=(a_v^{[j]})^2(1-e^{-h_j(v)})$ identifies the analytic curvatures with
the sequence curvatures and proves positivity on the indicated ranges.
\end{proof}

\subsection{Boundary estimates and completion}
For $n\ge1024$ it remains to consider distances $d\le47$.
We prove the required inequalities there, and for the complete sequences
with $n\le1023$, by exact rational bounds.

\begin{proposition}\label{prop:row-certificates}
For $n\ge1024$,
\[
 \begin{gathered}
 a_v^{[j]}>0\quad
 (0\le j\le8,\ 0\le v\le n,\ \min(v,n-v)\le56-j),\\
 s_8(k)>4/5\quad(1\le\min(k,n-k)\le47).
 \end{gathered}
\]
If $\min(k,n-k)=1,2,3,4$, respectively, then $s_T(k)>2/3$
for $T=2,2,3,4$.
For $2\le n\le1023$, there is an integer
$0\le t_n\le\min\{9,\lfloor\log_2n\rfloor\}$ such that
$a_v^{[j]}>0$ for $0\le j\le t_n$, $0\le v\le n$, and
$s_{t_n}(k)>2/3$ for $1\le k<n$.
\end{proposition}
\begin{proof}[Computer-assisted proof]
For the uniform edge bounds, set
\[
 \kappa_n=\frac1{4\sqrt n}\prod_{i=0}^{n-1}\frac{4i+5}{4i+3},
 \qquad
 R(v)=\frac{(4+v)(128+19v^2)}{4(128+21v^2)},\qquad
 g=\frac{\kappa_{1024}}{R(1/1024)}\in\mathbb Q.
\]
Setting $k=0,1$ in \eqref{eq:coefficients} and summing over $n$, we obtain
\begin{align*}
 \sum_{n\ge0}d_0(n)t^n&=(1-2t)^{-3/4},\\
 \sum_{n\ge0}d_1(n)t^n&=
 \frac34(1-2t)^{-7/4}-\frac12(1-2t)^{-5/4}
                         -\frac14(1-2t)^{-3/4}.
\end{align*}
Equivalently,
$d_0(n)=2^n(3/4)_n/n!$ and
$d_1(n)=(n+1/2)d_0(n)-2^{n-1}(5/4)_n/n!$.
Thus the definition of $p_0$ gives
\[
 p_0=\frac12\left(\frac{(5/4)_n}{(3/4)_n}-1\right)
     =2\sqrt n\,\kappa_n-1/2.
\]
For every integer $n\ge1024$ we use the enclosure
\begin{equation}\label{eq:row-edge-seed-enclosure}
 \left|\kappa_n-gR(1/n)-2^{-72}\right|\le2^{-72}.
\end{equation}
To prove this enclosure for all such indices, put
$y_n=\kappa_n/R(1/n)$. Then
\[
 \left(\frac{y_{n+1}}{y_n}\right)^2-1=\frac{N(n)}{D(n)},
\]
where
\begin{align*}
 N(n)&=44285952n^4+88571904n^3+58190848n^2
                  +13904896n+8014561,\\
 D(n)&=n(4n+3)^2(128n^2+21)^2(128n^2+256n+147)^2.
\end{align*}
For $n\ge1024$, the positive coefficients give
\[
 \frac{N(n)}{n^4}\le44285952+\frac{88571904}{1024}
 +\frac{58190848}{1024^2}+\frac{13904896}{1024^3}
 +\frac{8014561}{1024^4}<44\,400\,000.
\]
Also $D(n)>16\cdot128^4n^{11}$. Since
$44\,400\,000/(16\cdot128^4)<11/1000$, we obtain
$0<N(n)/D(n)<11/(1000n^7)$. Consequently
\[
 0\le\log(y_n/g)<a:=\frac{11}{2000}
                  \left(1024^{-7}+\frac1{6\cdot1024^6}\right).
\]
Evaluating the exact product at $1024$, we find $0<g<1/2$, and
$R(1/n)\le4097/4096$.  Since $e^a-1\le a/(1-a)$,
\[
 0\le\kappa_n-gR(1/n)
 <\frac12\frac{4097}{4096}\frac a{1-a}<2^{-71}.
\]
Thus \eqref{eq:row-edge-seed-enclosure} holds.

Put $u=n^{-1/2}$, $v=u^2$, $W_k=up_k$, and
$V_k=(p_{n-k}-(n-k))/v$. By the recurrence for $p_k$, we have
\begin{align*}
 W_0&=2\kappa_n-u/2,&
 W_{k+1}&=\frac{(1+v)W_k-k(k+1)uv}{1-uW_k},\\
 V_0&=0,&
 V_k&=\frac{k(1+vV_{k-1})}{2+(2-k)v+v^2V_{k-1}}.
\end{align*}
The two initial edge arrays have $s_0(0)=s_0(n)=1$ and
\begin{align*}
 s_0(k)&=\frac1{k+1}\left(1+
 \frac{kv\{2kv+(W_k-ku)(W_k+(k+1)u)\}}
 {(1+kv)(1-(k-1)v)}\right),\\
 s_0(n-k)&=\frac1{k+1}\left(1+
 \frac{v\{k+(1-kv)V_k(2+(1-2k)v+v^2V_k)\}}
 {(1-(k-1)v)(2-kv)}\right).
\end{align*}
The formulas extend continuously to $u=0$. We cover $0\le u\le1/32$
by the eight substitutions
\[
 u=\frac{2i+1+t}{512},\qquad -1\le t\le1,\quad 0\le i\le7.
\]
For each substitution we use monomial models of degree $20$, with
coefficients rounded to multiples of $2^{-512}$ and every discarded
coefficient included in the residual bound of
Lemma~\ref{lem:model-arithmetic}.  Starting with both arrays through
$k=55$, we retain the original endpoint value $s_j(0)=1$ and apply
\[
 s_{j+1}(k)=1-\left(\frac{1-s_j(k)}{s_j(k)}\right)^2s_j(k-1)s_j(k+1).
\]
The same update is applied to the reflected array $s_j(n-k)$.
At depth $j$ the retained range on either edge is $0\le k\le55-j$.
We verify that every denominator has a positive lower bound and
that the enclosures certify $0<s_j(k)\le1$ at every retained position
of each array. After eight
iterations, the lower bounds at $1\le k\le47$ exceed $994/1000>4/5$
throughout each parameter interval. From the same models, we obtain
the earlier bounds
\[
\begin{array}{c|cccc}
 d&1&2&3&4\\ \hline
 T&2&2&3&4\\
 \text{strict lower bound for }s_T(d),\ s_T(n-d)&6/7&7/10&4/5&3/4
\end{array}
\]
with every required predecessor retained.

For $2\le n\le1023$, we compute the initial ratios from the exact
integers $4^n d_k(n)$ given by \eqref{eq:coefficients}. We verify the stated inequalities
by interval arithmetic with outward rounding, using the full sequence
with both original endpoints retained. For each sequence, the computation based on the $p_k$ recurrence gives
the same number of iterations as the direct coefficient computation.
\end{proof}

\begin{proof}[Proof of Theorem~\ref{thm:rows}]
The initial sequence is positive by \eqref{eq:coefficients}. Its endpoints
are repeatedly squared, so the conclusion also holds for $n=0,1$.  For $n\ge2$ set
$K=\lfloor\log_2n\rfloor+1$. For $2\le n\le1023$, we have $t_n<K$
by Proposition~\ref{prop:row-certificates}.

For $n\ge1024$, put $d=\min(k,n-k)$.
If $1\le d\le47$, Proposition~\ref{prop:row-certificates} gives
$a_k^{[j]}>0$ for $0\le j\le8$ and $c_8(k)<1/3$.
For $d\ge48$, take
\[
 t_k=\lfloor\log_2d\rfloor+2\le\lfloor\log_2n\rfloor+1=K,
\]
where the inequality follows from $d\le n/2$.
By Proposition~\ref{row:lem:interior}, $a_k^{[j]}>0$ for
$0\le j\le t_k$ and $c_{t_k}(k)<1/3$.
Since $8<K$, Lemma~\ref{lem:local-stop} with $q=1/3$ gives
strict positivity of all iterates and $c_j(k)\le1/3$ for every $j\ge K$.
\end{proof}

\section{Transposed Boros--Moll sequences}\label{sec:tails}
For a fixed nonnegative integer $\ell$, consider the transposed sequence
$(d_\ell(\ell+k))_{k\ge0}$.
Jiang and Wang \cite[Theorems~3.1--3.3]{JiangWang} studied the
sequences obtained by fixing the coefficient index and varying the
degree. They proved log-convexity for $\ell=0$, log-concavity for
$\ell\ge1$, and $2$-log-concavity for $\ell\ge2$. They also established
the higher-order Tur\'an inequality for $\ell\ge2$, and asked how
the signs of the iterates of $\LL$ depend on the coefficient index
and the iteration depth. This higher-order inequality is also called the cubic Newton inequality
and is distinct from a condition on all iterates of $\LL$.

Zhao \cite{ZhaoPublished} obtained more precise bounds for the
adjacent ratios and proved extended reverse ultra log-concavity. In our notation,
for $\ell\ge1$ and $n\ge\ell+1$, his bounds are
\[
 \frac{(n-\ell+1)n^3}{(n-\ell)(n+1)(n^2+1)}
 <\frac{d_\ell(n)^2}{d_\ell(n-1)d_\ell(n+1)}
 <\frac{(n-\ell+1)n}{(n-\ell)(n+1)}.
\]
By the upper bound,
$(d_\ell(n)/\binom n\ell)_{n\ge\ell}$ is strictly log-convex;
the lower bound is stronger than ordinary log-concavity. For fixed
$\ell$, both bounds tend to $1$ as $n\to\infty$, so the initial
curvature tends to zero along the transposed sequence.

Motivated by the infinite log-concavity conjecture of Boros and Moll,
Zhao \cite{ZhaoPublished} conjectured that $(d_\ell(n))_{n\ge\ell}$
is strictly infinitely log-concave for every $\ell\ge3$. With
$n=\ell+k$, this means that all iterates of the sequence
$(d_\ell(\ell+k))_{k\ge0}$ are positive at every index, using zero
extension to the left. We prove this conjecture in the following theorem.

\begin{theorem}\label{thm:main}
For every integer $\ell\ge3$, the sequence
$a_k=d_\ell(\ell+k)$, $k\ge0$, is strictly infinitely log-concave.
\end{theorem}

Using the moment formula in Subsection~\ref{sec:integrals}, we write
each term as a Gamma factor times an integral and estimate the
initial curvature. We prove the required inequalities analytically
in the tail and by exact rational bounds at the remaining
initial indices.

\subsection{Integral representation and uniform estimates}\label{subsec:uniform-tail}
Write $(\xi)_r=\xi(\xi+1)\cdots(\xi+r-1)$, with $(\xi)_0=1$, and let $\B$ denote
the beta function.
Replace \(n,k\) in the right side of \eqref{eq:moment} by real numbers
\(y,\lambda\), with \(y\ge\lambda>0\), and denote the resulting positive
function by \(D_\lambda(y)\). This function interpolates the coefficients.
For fixed $\lambda$, we consider the sequence
$a_k=D_\lambda(\lambda+k)$, $k\ge0$, with zero extension to the left.
The substitution \(u=x^2\) and the duplication and reflection formulas
give
\begin{equation}\label{eq:factorization}
D_\lambda(y)=\mathcal B_\lambda(y)Q_\lambda(y+\lambda+1),
\end{equation}
where
\begin{align}
\mathcal B_\lambda(y)
&=\frac{2^y\Gamma(y+3/4)\Gamma(y+5/4)\Gamma(y+\lambda+1)}
{\Gamma(3/4)\Gamma(\lambda+1)\Gamma(y-\lambda+1)
 \Gamma(y+1)\Gamma(y+\lambda+5/4)},\label{eq:gamma-factor}\\
Q_\lambda(\xi)
&=\frac1{\B(1/4,\xi)}
\int_0^1u^{-3/4}(1-u)^{\xi-1}(1+\sqrt u)^{-2\lambda}\,du.
\label{eq:Q-integral}
\end{align}

Lagrange inversion gives
\begin{equation}\label{eq:Q-series}
Q_\lambda(\xi)=\sum_{r\ge0}w_{\lambda,r}
\frac{(\xi)_r}{(\xi+1/4)_r},
\qquad
w_{\lambda,r}=4^{-\lambda-r}\frac{\lambda}{\lambda+r}
\binom{2\lambda+2r}{r},
\qquad \sum_rw_{\lambda,r}=1.
\end{equation}

We use positivity of the weights to obtain an integral representation
for the logarithm.

\begin{lemma}\label{lem:strips}
Let $0<b<1$, $w_r>0$, and $\sum_{r\ge0}w_r=1$.  If
$Q(\xi)=\sum_{r\ge0}w_r(\xi)_r/(\xi+b)_r$, then there exist
$\alpha_i$, $i\ge0$, with $i<\alpha_i<i+b$, such that, for $\Rea \xi>0$,
\begin{equation}\label{eq:strips}
 -\log Q(\xi)=\sum_{i\ge0}\int_{\alpha_i}^{i+b}\frac{du}{\xi+u}
            =\int_0^\infty e^{-\xi t}\eta(t)\,dt,
\end{equation}
with the logarithm real on the positive axis, and
\begin{equation}\label{eq:eta-bound}
 0<t\eta(t)<\frac{1-e^{-bt}}{1-e^{-t}}\qquad(t>0).
\end{equation}
\end{lemma}
\begin{proof}
Write $f_r(\xi)=(\xi)_r/(\xi+b)_r$,
$Q_N=\sum_{r=0}^Nw_rf_r$ and $S_N=\sum_{r=0}^Nw_r$.
The identity
\[
 1-f_r(\xi)=\sum_{i=0}^{r-1}
 \frac{(b)_{i+1}(1-b)_{r-i-1}}{i!(r-i-1)!}\frac1{\xi+b+i}
\]
implies that $Q_N$ is strictly increasing between consecutive poles. In
$(-i-b,-i)$ it goes from negative infinity to
$Q_N(-i)=\sum_{r=0}^iw_rf_r(-i)>0$.  Its unique zero there is
$-\alpha_{N,i}$, and these are all its $N$ zeros.  Hence
\[
 \frac{Q_N(\xi)}{S_N}
 =\prod_{i=0}^{N-1}\frac{\xi+\alpha_{N,i}}{\xi+i+b}.
\]
On this interval $f_{N+1}<0$, so adding the next summand moves the
zero to the right.  Thus $\alpha_{N,i}$ decreases with $N$.

To pass to the limit, observe that the series defining $Q$ converges
locally uniformly away from its poles. On any such compact set, all
factors $(\xi+\ell)/(\xi+b+\ell)$ with sufficiently large index
have modulus at most one, while the finitely
many earlier factors are uniformly bounded. The claim follows from
$\sum w_r=1$.
In particular, $Q(-i)=\sum_{r=0}^iw_rf_r(-i)>0$. By local uniform
convergence, the limiting zero is distinct from $-i$. The strict bound
at the other endpoint already holds for each finite truncation.
Consequently $\alpha_{N,i}\downarrow\alpha_i$ with
$i<\alpha_i<i+b$.

For real $\xi>0$, the values $f_r(\xi)$ decrease with $r$, so
$Q_N(\xi)/S_N$ decreases to $Q(\xi)>0$.  Taking logarithms of the finite
product and applying monotone convergence gives
\[
 -\log Q(\xi)=\sum_{i\ge0}\int_{\alpha_i}^{i+b}\frac{du}{\xi+u}<\infty.
\]
The sum converges locally uniformly in $\Rea \xi>0$, since on
$\Rea \xi\ge\sigma>0$ it is dominated by its value at $\sigma$.
Exponentiating and applying the identity theorem, we obtain the
asserted holomorphic logarithm throughout that half-plane.

Set $\eta(t)=\sum_{i\ge0}\int_{\alpha_i}^{i+b}e^{-ut}\,du$.
By Tonelli's theorem, the Laplace representation holds for real $\xi>0$;
by absolute convergence, it also holds for $\Rea \xi>0$.
Enlarging each integration interval to $[i,i+b]$ proves
\eqref{eq:eta-bound}, with strict inequalities because
$i<\alpha_i<i+b$.
\end{proof}

\begin{lemma}\label{lem:kernel}
Put \(H_\lambda(y)=D_\lambda(y)/(2^y\binom y\lambda)\).  There is a
kernel \(k_\lambda\) such that
\begin{equation}\label{eq:kernel}
-(\log H_\lambda)'(y)=\int_0^\infty e^{-yt}k_\lambda(t)\,dt,
\qquad 0<k_\lambda(t)<\frac14e^{-3t/8}.
\end{equation}
Consequently, for \(e_\lambda=\D\log H_\lambda\),
\begin{equation}\label{eq:e-bound}
e_\lambda(y)=\int_0^\infty e^{-yt}
\frac{2\cosh t-2}{t}k_\lambda(t)\,dt,
\qquad
|e_\lambda(y)|<\frac1{4(\Rea y-5/8)^2}
\quad(\Rea y>1).
\end{equation}
\end{lemma}

\begin{proof}
Apply Lemma~\ref{lem:strips} to \eqref{eq:Q-series} with \(b=1/4\).
Writing \(u=e^{-t/4}\), logarithmic differentiation of
\eqref{eq:gamma-factor} gives
\[
k_\lambda(t)=\frac{u^3-u^4+u^{4\lambda+4}}{1+u+u^2+u^3}
-tu^{4\lambda+4}\eta_\lambda(t).
\]
By the strip bound, \(0<t\eta_\lambda<1/(1+u+u^2+u^3)\), and hence
\(0<k_\lambda<u^3/(1+u+u^2+u^3)\le u^{3/2}/4\).
Taking a second difference proves
the first formula in \eqref{eq:e-bound}. The corresponding bound for
\((\log H_\lambda)''\) implies the last inequality.
\end{proof}

We will also need bounds for the Taylor coefficients of the integral
correction. The following estimate is independent of the measurable
function in the integrand.

\begin{lemma}
\label{lem:normalized-laplace-coefficients}
Let $0\le f(t)\le1$ be measurable and put
$F(v)=v^2\int_0^\infty te^{-vt}f(t)\,dt$ for $\Rea v>0$.
For every $v_0>0$, write $F(v_0(1+\xi))=\sum_{i\ge0}F_i\xi^i$.
Then $0\le F_0\le1$ and $|F_i|\le1$ for $i\ge1$.
\end{lemma}
\begin{proof}
After the change of variables $u=v_0t$, the coefficient for $i\ge1$ is
\[
 F_i=\frac{(-1)^i}{i!}\int_0^\infty u^{i-1}e^{-u}
       \bigl((u-i)^2-i\bigr)f(u/v_0)\,du.
\]
The kernels $u^{i-1}e^{-u}(u-i)^2$ and $iu^{i-1}e^{-u}$ both integrate
to $i!$. Since $0\le f\le1$, the difference of their weighted integrals lies between $-i!$ and $i!$.
This proves the bound. The constant term is immediate, and dominated
convergence on $|\xi|\le r<1$ justifies the expansion.
\end{proof}
Since the factor \(2^y\) has zero logarithmic second difference, the
initial shifted curvature is
\begin{equation}\label{eq:initial-curvature}
h_{\lambda,0}(z):=-\D\log D_\lambda(\lambda+z)
=B_\lambda(z)-e_\lambda(\lambda+z),
\qquad
B_\lambda(z)=\log\frac{(z+1)(z+\lambda)}{z(z+\lambda+1)}.
\end{equation}
We use \eqref{eq:curvature-iteration} for the later curvatures. The
Stieltjes representation is required only for the initial curvature.

To obtain the asymptotic expansion in Table~\ref{tab:five-applications}, we use
\eqref{eq:small-Gamma} and \eqref{eq:EO}, which hold for every
fixed integer $\ell\ge1$. By the recurrence, $E_\ell(0)=1$ and
$O_\ell(0)=2\ell$, so the bracket in \eqref{eq:small-Gamma} is
$1+O_\ell(|z|^{-1/2})$ on closed subsectors of $\Rea z>0$.
Using Lemma~\ref{lem:Gamma-remainder} and Cauchy's estimate, we bound
the second difference of its logarithm by $O_\ell(|z|^{-5/2})$.
Applying the Gamma functional equation, we obtain
\[
 h_{\ell,0}(z-\ell)
 =-\ell\Log(1-z^{-2})
  +\Log\frac{(z+1)(z-1/4)}{z(z+3/4)}
  +O_\ell(|z|^{-5/2}).
\]
After translating back to the sequence coordinate, this becomes
\[
 h_{\ell,0}(z)=(\ell-1/4)z^{-2}+O_\ell(|z|^{-5/2}).
\]
For each fixed $\ell\ge3$, we consequently obtain, by
Theorem~\ref{thm:power-law},
\[
 T_{q_*}(k)=2\log_2k-\log_2(\ell-9/4)+O_\ell(1).
\]

To bound the iterated correction, we use an absolute coefficient norm.
For $f(\xi)=\sum_{i\ge0}f_i\xi^i$, put
$\pnorm f r=\sum_{i\ge0}|f_i|r^i$.
For power series write \(f\preceq U\) if every absolute Taylor coefficient
of \(f\) is at most the corresponding nonnegative coefficient of \(U\).
Then \(|f(\xi)|\le U(r)\) on \(|\xi|\le r\).  For \(0<h\le H\), put
\[
\mathcal T_hf(\xi)=\frac{f(\xi-h)-2f(\xi)+f(\xi+h)}{h^2}.
\]
If \(f\preceq U\), then
\begin{equation}\label{eq:T-majorant}
\mathcal T_hf\preceq\mathcal T_HU,
\qquad
\mathcal T_HU=2\sum_{i\ge1}\frac{H^{2i-2}}{(2i)!}U^{(2i)}.
\end{equation}
In particular, \(\mathcal T_H\) preserves nonnegative coefficients.

Using the second curvature in \eqref{eq:two-profiles}, we estimate
the tail uniformly for coefficient indices $\lambda\ge19$.
For real $\lambda\ge19$ and integer $k\ge256$, we keep both denominator factors in
\begin{equation}\label{eq:uniform-W}
 W_\lambda(z)=\frac{\lambda}{z(z+\lambda+1)},
 \qquad B_\lambda(z)=\log(1+W_\lambda(z)).
\end{equation}
Put
\begin{equation}\label{eq:uniform-parameters}
 \delta=\frac1k,\quad \omega=\frac{k}{k+\lambda+1},\quad
 \chi=\lambda\omega>35/2,\quad z=k(1+\xi),\quad P(\xi)=(1+\xi)(1+\omega \xi).
\end{equation}
Thus $W_\lambda(z)=\chi\delta^2/P(\xi)$. Normalize the curvatures by
\begin{equation}\label{eq:uniform-normalization}
 h_{\lambda,j}(k(1+\xi))=2^jW_\lambda(k(1+\xi))(\beta_{\lambda,j}+R_j(\xi)),
 \qquad \beta_{\lambda,j}=1-\frac{2-2^{1-j}}\lambda.
\end{equation}
The identity
\begin{equation}\label{eq:two-pole-identity}
 P(\xi)\left\{\frac1{(1+\xi)^2}+\frac{\omega^2}{(1+\omega \xi)^2}\right\}
 =2\omega+\frac{(1-\omega)^2}{P(\xi)}
\end{equation}
separates the constant drift from the remainder, for which
\begin{equation}\label{eq:uniform-small-drift}
 \frac{(1-\omega)^2}{\chi}\le\frac{1+1/\lambda}{k}<\frac9{8k}.
\end{equation}
By Lemma~\ref{lem:reference-scale}, the exact recurrence is
\begin{align}\label{eq:uniform-exact}
 R_{j+1}={}&R_j+\frac{2^{-j}}\lambda
 +\frac{P}{2^{j+1}\chi}\mathcal T_\delta\log P\notag\\
 &-\frac{P}{2^{j+1}\chi}\mathcal T_\delta
 \{\log(\beta_{\lambda,j}+R_j)+\phi(h_{\lambda,j})\}.
\end{align}

\begin{proposition}\label{prop:uniform-majorant}
Put $r=11/50$ and $J=\min\{j\ge0:2^j\chi\ge2240\}\le7$.
The curvature $h_{\lambda,j}$ is holomorphic with positive real part
near $|z-k|\le k(r+(7-j)/256)$ for $0\le j\le J$. Moreover,
\begin{equation}\label{eq:uniform-terminal}
 |R_J(\xi)|<
 \begin{cases}
  1/3,& |\xi|\le r,\\
  1/10,& |\xi|\le3r/4,\\
  3/100,& \xi=0.
 \end{cases}
\end{equation}
\end{proposition}
\begin{proof}
Put $H=1/256$. We begin with the initial majorant
\begin{equation}\label{eq:uniform-U0}
 U_0(\xi)=\frac{H}{2((1-\xi)^2-H)}+\frac1{76(1-\xi)}
       +\frac{9H}{32(1-\xi)^2}+\frac{7H^2}{16(1-\xi)^3}.
\end{equation}
The first term bounds $B_\lambda/W_\lambda-1$, since
$W_\lambda\preceq H(1-\xi)^{-2}$ and
$\log(1+v)/v-1\preceq v/(2(1-v))$.
For the kernel term, write $k_\lambda(t)=\tfrac14e^{-3t/8}f_\lambda(t)$,
where $0<f_\lambda<1$ by Lemma~\ref{lem:kernel}. Set
\begin{align*}
 G_a(z)&=\frac{z(z+\lambda+1)}{\lambda(z+a)^2}
 =\frac1\lambda+\frac{\lambda+1-2a}{\lambda(z+a)}
            +\frac{a(a-\lambda-1)}{\lambda(z+a)^2},\\
 F_a(z)&=(z+a)^2\int_0^\infty te^{-(z+a)t}f_\lambda(t)\,dt.
\end{align*}
The integral form of the second difference gives
\[
 \frac{e_\lambda(\lambda+z)}{W_\lambda(z)}
 =\frac14\int_{-1}^1(1-|v|)
 G_{\lambda+v+3/8}(z)F_{\lambda+v+3/8}(z)\,dv.
\]
For $a=\lambda+v+3/8$ and $|v|\le1$, applying Lemma~\ref{lem:normalized-laplace-coefficients} at $k+a$,
we obtain $F_a(k(1+\xi))\preceq(1-\xi)^{-1}$, since $k/(k+a)<1$.
Also, since $-5/8\le a-\lambda\le11/8$, we have
\[
 G_a(k(1+\xi))\preceq\frac1{19}+\frac{9H}{8(1-\xi)}
                                      +\frac{7H^2}{4(1-\xi)^2}.
\]
The triangular weight integrates to $1$, so $R_0\preceq U_0$.

Define
\[
 \Lambda(v)=v+\frac{v^2}{2(1-v)},\qquad
 \Xi(X)=\frac{X}{2(1-35X/6144)},\qquad X_j(\xi)=\frac{1+U_j(\xi)}{(1-\xi)^2},
\]
\[
 A(\xi)=\frac{9H}{8(1-\xi)^2}
 +\frac{2H^2(1+\xi)^2}{35(1-\xi)^2((1-\xi)^2-H^2)},
\]
and define the rational functions with nonnegative coefficients recursively by
\begin{align}\label{eq:uniform-majorant}
 U_{j+1}={}&U_j+\frac{A}{2^{j+1}}
 +\frac{(1+\xi)^2}{35\cdot2^j}\mathcal T_H\Lambda(19U_j/17)\notag\\
 &+\frac{H^2(1+\xi)^2}{2}\mathcal T_H\Xi(X_j).
\end{align}
Indeed, the first two forcing terms in \eqref{eq:uniform-exact} equal
\[
 -\frac1{2^{j+1}}\left\{\frac{(1-\omega)^2}{\chi P}
 +\frac P\chi\sum_{i\ge2}\frac{\delta^{2i-2}}i
 \left(\frac1{(1+\xi)^{2i}}+\frac{\omega^{2i}}{(1+\omega \xi)^{2i}}\right)\right\}.
\]
Using \eqref{eq:uniform-small-drift}, $2/i\le1$, $\chi>35/2$ and $\delta\le H$
gives the bound $A/2^{j+1}$. Since $\mathcal T_\delta$ annihilates constants
and $\beta_{\lambda,j}\ge17/19$, the remaining logarithm is bounded by
$\Lambda(19U_j/17)$. For $j<J$,
\[
 \kappa_j=2^j\chi/k^2<2240/256^2=35/1024.
\]
The even coefficients of $\phi$ have absolute values
$\zeta(2i)/(i(2\pi)^{2i})<2/6^{2i}$, whence
\[
 \phi(v)\preceq\frac v2+\frac{2v^2}{36-v^2}
 \preceq\frac{v}{2(1-v/6)},\qquad
 \frac{\phi(\kappa_jX_j)}{\kappa_j}\preceq\Xi(X_j).
\]
Combining these bounds with $P\preceq(1+\xi)^2$ and
\eqref{eq:T-majorant}, we obtain $R_{j+1}\preceq U_{j+1}$ whenever
the logarithmic series converge.

At each of the centers $\xi_*=r,3r/4,0$, we evaluate
\eqref{eq:uniform-majorant} at $\xi_*+mH$ for integers $-7\le m\le7$,
retaining $|m|\le7-j$ at iteration $j$. With outward rounding to
multiples of $2^{-256}$, we verify positivity of the denominators and
\begin{equation}\label{eq:uniform-tests}
 U_j(r+(7-j)H)<21/100\quad(j<7),\qquad
 U_7(r)<1/3,\quad U_7(3r/4)<1/10,\quad U_7(0)<3/100.
\end{equation}
The calculation consists of $3\sum_{j=0}^{7}(15-2j)=192$ rational
enclosures, obtained from \eqref{eq:uniform-U0} and
\eqref{eq:uniform-majorant}.
These bounds imply $U_j<17/19$ and $35X_j<6144$ on the relevant disks,
so every logarithmic series converges and every rational composition
is well defined there.
Since $U_{j+1}-U_j$ has nonnegative coefficients, the final bounds
also hold when $J<7$.

For $0\le v\le1/2$, we use the representation
$W_\lambda(z)=\lambda(\lambda+1)^{-1}\int_0^{\lambda+1}(z+t)^{-2}\,dt$
and the inversion estimate in Corollary~\ref{cor:normalized-stop}
to obtain
\[
 \Rea W_\lambda(z)\ge\frac{W_\lambda(k)}{(1+v)^2},\qquad
 |W_\lambda(z)|\le\frac{W_\lambda(k)}{(1-v)^2}
 \quad(|z-k|\le vk).
\]
All preliminary radii are less than $k/4$. Since
\[
 \frac{17}{19(5/4)^2}-\frac{21/100}{(3/4)^2}>0,
 \qquad \frac{17}{19(1+r)^2}-\frac{1/3}{(1-r)^2}>0
\]
the curvatures have positive real part throughout the induction.
Each radius loss is $kH\ge1$. By the strict bounds on the closed
disks, the functions are holomorphic on neighborhoods of these disks
even when $kH=1$.
\end{proof}

\begin{proposition}\label{prop:uniform-tail}
For every real $\lambda\ge19$ and integer $k\ge256$, there is
an integer $T=T(\lambda,k)\ge0$ such that
\[
 a_v^{[j]}>0\quad
 (0\le j\le T,\ v\ge0,\ |v-k|\le T+1-j),
 \qquad c_T(k)\le q_*.
\]
\end{proposition}
\begin{proof}
Put $r=11/50$ and $Q=2^J\chi$. By the preceding estimates,
\[
 \Rea h_{\lambda,J}>0\quad(|z-k|\le rk),\qquad
 \Rea h_{\lambda,J}>\frac{Q}{2k^2}\quad(|z-k|\le3rk/4),\qquad
 h_{\lambda,J}(k)>\frac{43Q}{50k^2},
\]
since $17/[19(1+3r/4)^2]-(1/10)/(1-3r/4)^2>1/2$ and
$17/19-3/100>43/50$. Take $K=\lceil\log_2(3k^2/Q)\rceil$ and
radii $rk(3/4)^i$. Here $K\ge1$: if $J>0$, then $Q<4480$,
whereas if $J=0$, then $Q=\chi<k$.
The smallest squared gap exceeds five. For $K\le7$ and $K\ge8$,
respectively, this follows from
\[
 \frac{256^2r^2}{16}(9/16)^6>5,\qquad
 \frac{2240r^2}{48}(9/8)^7>5,
\]
the second comparison using $k^2>Q2^{K-1}/3$.
By Corollary~\ref{cor:geometric-budget}, $\mathcal E_i<1085/k^2$.
Every disk after the first lies in $|z-k|\le3rk/4$, where
$Q/2\ge1120>1085$. We can therefore apply
Theorem~\ref{thm:unified-entry} to obtain
\[
 h_{\lambda,J+K}(k)>
 3\left(\frac{43}{50}-\frac{1085}{2240}\right)>1>\rho_*.
\]
All unit shifts lie in the preceding disks. Thus the curvatures agree
with those of the original sequence, as in
Section~\ref{subsec:prepared-entry}.
\end{proof}

\subsection{The range \texorpdfstring{$3\le\ell\le18$}{3 <= ell <= 18}}
\label{subsec:small-tail}

For $3\le\ell\le18$, we instead expand the moment representation
in powers of $z^{-1/2}$ and control the resulting comparison model
through the preliminary iterations.

Fix an integer $3\le\ell\le18$. The original degree is $n=\ell+k$.
In this subsection, $z$ is the complex variable corresponding to the degree,
so the sequence coordinate is $z-\ell$. Put $b_\ell=\ell-9/4$ and
\[
 \beta_{\ell,j}=b_\ell+2^{1-j},\qquad
 J_\ell=\left\lceil\log_2\frac{3072}{b_\ell}\right\rceil\le12,\qquad
 Z_{\ell,j}(z)=2^{-j}z^2h_{\ell,j}(z-\ell).
\]
On $\Rea z>0$, we use $z^{\pm1/2}=\exp(\pm\tfrac12\Log z)$,
where $\Log$ is the principal logarithm. For the comparison model, write
\[
 G(z)=\log\frac{\Gamma(z+1/4)z^{1/2}}{\Gamma(z+3/4)},\quad
 \gamma=\frac{\Gamma(3/4)}{\Gamma(1/4)},\quad
 V(u)=\sum_{i=1}^8\frac{\mathsf E_{2i}}{4i16^i}u^{4i},
\]
where $\operatorname{sech}t=\sum_{i\ge0}\mathsf E_{2i}t^{2i}/(2i)!$.
The logarithm defining $G$ is holomorphic on $\Rea z>0$ and real on the positive axis.
\begin{lemma}\label{lem:Gamma-remainder}
For $x>0$ and $\Rea z>0$,
\[
 -\frac1{64x^2}\le G(x)\le0,\qquad
 |G(z)-V(z^{-1/2})|\le (\Rea z)^{-18}.
\]
\end{lemma}
\begin{proof}
From the digamma integral, we have
$G(z)=\tfrac12\int_0^\infty e^{-zt}(\operatorname{sech}(t/4)-1)\,dt/t$.
By the product for $\cosh$ \cite[Eq.~4.36.2]{DLMF},
$\operatorname{sech}(\sqrt u/4)$ is a locally uniform limit of products
of completely monotone factors $(1+c u)^{-1}$, $c>0$. Its derivatives
therefore alternate in sign on $u\ge0$.
Its signed Taylor remainder is bounded by the first omitted
term. Integration gives both assertions, since
$|\mathsf E_{18}|/(36\cdot16^9)<1$.
\end{proof}

Integration by parts in the moment formula gives
\begin{equation}\label{eq:small-Gamma}
 D_\ell(z)=\frac{2^zz^\ell\Gamma(z+3/4)}{\ell!\Gamma(3/4)\Gamma(z+1)}
 \{E_\ell(z^{-1})-\gamma z^{-1/2}O_\ell(z^{-1})e^{G(z)}\}.
\end{equation}
Here $E_0=1$, $E_1=1+y/2$, $O_0=0$, $O_1=2+y/2$, and both families satisfy
\begin{equation}\label{eq:EO}
 f_{i+2}(y)=(2+y)f_{i+1}(y)-\{1+y-i(i+1)y^2\}f_i(y).
\end{equation}
Their coefficients are nonnegative and increase with $i$, as follows by
writing the recurrence for $f_{i+2}-f_{i+1}$.
For $u=z^{-1/2}$, define
\[
 A(u)=u^{-4}\log(1-u^4),\qquad
 \mathcal Sf=u^{-4}\left\{
 f\left(\frac{u}{\sqrt{1-u^2}}\right)
 +f\left(\frac{u}{\sqrt{1+u^2}}\right)-2f(u)\right\}.
\]
The square roots are the holomorphic branches equal to $1$ at $u=0$.
Replacing $G$ by $V$ in the initial formula gives the analytic model
\begin{align}\label{eq:wiener-initial-model}
 \widehat Z_{\ell,0}(u)={}&-\ell A(u)
 -u^{-4}\log\frac{1+3u^2/4}{(1-u^2/4)(1+u^2)}\notag\\
 &-\mathcal S\log\{E_\ell(u^2)-\gamma uO_\ell(u^2)e^{V(u)}\}.
\end{align}
The second term follows from the Gamma functional equation. We choose
the logarithms by their values at $u=0$. The apparent singularities are
removable, since
\[
 A(u)=-1+O(u^4),\qquad
 \log\frac{1+3u^2/4}{(1-u^2/4)(1+u^2)}
 =\frac14u^4+O(u^6),
\]
and $\mathcal S$ maps a holomorphic power series to a holomorphic
power series with zero constant term. Thus
$\widehat Z_{\ell,0}(0)=\ell-1/4=\beta_{\ell,0}$.
We estimate the Gamma remainder separately on disks in the $z$-plane
and define the subsequent approximations recursively by
\begin{equation}\label{eq:Z-recurrence}
 Z^+=Z+2^{-j}A-2^{-j-1}\mathcal S\{\log Z+\phi(2^ju^4Z)\}.
\end{equation}
Thus $T_{\ell,j}:=\widehat Z_{\ell,j}-\beta_{\ell,j}$ has zero constant
term at every layer.

We use one sequence of radii for all sixteen values of $\ell$. Put
\begin{equation}\label{eq:wiener-radii}
 r_0=\frac1{25},\qquad
 \theta_j=1-\frac15\left(\frac45\right)^j,\qquad
 r_{j+1}=\theta_jr_j,
\end{equation}
and set $C_j=3/(22(1-\theta_j-r_{j+1}^2)^2)$.
We continue to use the coefficient norm $\pnorm f r=\sum_{i\ge0}|f_i|r^i$.

\begin{lemma}\label{lem:analytic-wiener}
For $0\le j<J_\ell$, let $f(0)=0$ and $\pnorm{f}{r_j}<\infty$.
With $\kappa=(2^{j+1}\beta_{\ell,j})^{-1}$, we have
\[
 \pnorm{(I-\kappa\mathcal S)f}{r_{j+1}}
 \le\theta_j\pnorm{f}{r_j},\qquad
 \pnorm{\mathcal Sf}{r_{j+1}}\le C_j\pnorm{f}{r_j}.
\]
\end{lemma}
\begin{proof}
Write $R=r_j$, $r=r_{j+1}$, $\theta=\theta_j$,
$\vartheta=\theta/\sqrt{1-r^2}$, and $\nu_i=i(i+2)/4$.
The diagonal coefficient of $\mathcal Su^i$ is $\nu_i$, and all
remaining coefficients are nonnegative. Taylor's formula gives
\begin{equation}\label{eq:optimized-monomial-bounds}
 \frac{\pnorm{\mathcal Su^i}{r}}{R^i}
 \le\frac{\nu_i\vartheta^i}{(1-r^2)^2},\qquad
 \frac{\pnorm{\mathcal Su^i-\nu_i u^i}{r}}{r^i}
 \le\frac{\nu_i(i+4)(i+6)r^4}{48(1-r^2)^{i/2+4}}.
\end{equation}
For every integer $i\ge1$ and $0<\vartheta<1$,
\[
 \nu_i\vartheta^i(1-\vartheta)^2
 \le\left(\frac{i}{i+2}\right)^{i+1}<e^{-2}<\frac3{22}.
\]
The first inequality follows by maximizing in $\vartheta$; the second follows
from $\log((1+x)/(1-x))>2x$ for $0<x<1$.
Since $\vartheta\le\theta/(1-r^2)$, we have $\pnorm{\mathcal Su^i}{r}/R^i\le C_j$.

Put $K_j=(2^{j+1}(3/4+2^{1-j}))^{-1}$.
The explicit radii satisfy $r^2<1-\theta$ and $K_jC_j<\theta_j$
for $0\le j\le11$. At $\kappa=K_j$, if $K_j\nu_i>1$, then
$\pnorm{(I-K_j\mathcal S)u^i}{r}/R^i\le K_jC_j-\theta^i<\theta$.
Otherwise $i<110$; the second bound in
\eqref{eq:optimized-monomial-bounds}, with $r\le1/25$, is less
than $\nu_i/1000$. Hence
$\pnorm{(I-K_j\mathcal S)u^i}{r}/R^i\le\theta^i$.
Convexity in $\kappa\in[0,K_j]$ and absolute summation prove the
assertions. Since $\vartheta<1$, both composition arguments lie inside $|u|<R$.
\end{proof}

\begin{proposition}\label{prop:small-grid}
For $3\le\ell\le18$ and $0\le j\le J_\ell$, the model is holomorphic
near $|u|\le r_j$ and
\begin{equation}\label{eq:uniform-small-residual}
 \pnorm{T_{\ell,j}}{r_j}<4b_\ell r_j,
 \qquad |T_{\ell,j}(u)|<4b_\ell|u|\quad(0<|u|\le r_j).
\end{equation}
\end{proposition}
\begin{proof}
Put $t=r_0/(1-r_0^2)$ and
\[
 H_\ell(u)=E_\ell(u^2)-1+
 \frac{17uO_\ell(u^2)}{50(1-V_*(u))},\qquad
 V_*(u)=\sum_{i=1}^8\frac{|\mathsf E_{2i}|}{4i16^i}u^{4i}.
\]
Lemma~\ref{lem:Gamma-remainder} and
$12\prod_{i=0}^{143}(4i+1)/(4i+3)
 <(17/50)(1-1/(64\cdot144^2))$ give $\gamma<17/50$.
Thus $-\log(1-H_\ell)$ majorizes the initial logarithm.
With $H=H_\ell(t)$ and $H_i=[(u\,d/du)^iH_\ell](t)$,
\eqref{eq:optimized-monomial-bounds} gives
\[
 \pnorm{T_{\ell,0}}{r_0}\le
 \frac{\ell r_0^4}{2(1-r_0^4)}+\frac{r_0^2}{1-r_0^2}
 +\frac1{4(1-r_0^2)^2}
 \left(\frac{H_2+2H_1}{1-H}+\frac{H_1^2}{(1-H)^2}\right)
 <3b_\ell r_0.
\]
The last inequality and $H<1$ follow by rational substitution for
$3\le\ell\le18$. The full series is included in this bound.

Let $A_j=\max_{a\ge3/4}a/(a+2^{1-j})^2$;
the maximum occurs at $a=\max(3/4,2^{1-j})$. Define
\begin{align}\label{eq:optimized-budget-test}
 d_j={}&\frac{4C_jr_jA_j}{2^j\theta_j(1-4r_j)}
 +\frac{r_{j+1}^3}{2^{j+1}(3/4)(1-r_{j+1}^4)}\notag\\
 &+\frac{3C_jr_j^3}{10\theta_j}
       \left(1+\frac{2^{3-j}}{3}+4r_j\right).
\end{align}
These explicit rational numbers satisfy $\sum_{j=0}^{11}d_j<1$.
Assuming $\pnorm{T_{\ell,j}}{r_j}<4b_\ell r_j$, we apply
Lemma~\ref{lem:analytic-wiener}, the logarithmic remainder estimate,
and $\pnorm{\phi(f)}{r_j}\le3\pnorm f{r_j}/5$ to bound the increase of
$\pnorm{T_{\ell,j}}{r_j}/(b_\ell r_j)$ by $d_j$.
Induction therefore gives the bound $3+\sum_{i<j}d_i<4$.
All logarithmic series converge, since
$4r_j\le4/25$; before $J_\ell$, the $\phi$ argument is at most
$r_0^4\{3072(1+4r_0)+2\}<1/10$.
Finally, $T_{\ell,j}(0)=0$ gives the pointwise bound.
\end{proof}

For the comparison with the original curvature, fix $n\ge8192$ and put
$J=J_\ell$, $\widehat r_j=2/5+(J-j)/100$, and
$\mathcal D_j=\{|z-n|\le\widehat r_j n\}$.
The explicit radii satisfy $r_{12}>1/68$ and
$\theta_j<197/200$ for $0\le j\le11$.
Hence $(12/25+j/100)r_j^2$ decreases, and
$8192(1-\widehat r_j)r_j^2>1$ follows from $(3/5)8192>68^2$.
For each disk, the transformed arguments lie in the corresponding model domain, so
$|T_{\ell,j}(z^{-1/2})|<b_\ell/15$.

Let $e_j=\norm{Z_{\ell,j}-\widehat Z_{\ell,j}(z^{-1/2})}{\mathcal D_j}$.
On $|z-n|\le2n/3$, we have $|u|<1/50$,
$H_{18}(1/50)<1/3$, and $V_*(1/50)<1/1000$.
Lemma~\ref{lem:Gamma-remainder} bounds the initial logarithmic
difference by $10(n/3)^{-18}$. The second-difference estimate,
with a gap of at least $11n/75$, gives
\[
 e_0\le2^{12}(n/3)^{-18}<2^{-190}.
\]
We prove holomorphy, positivity of the real part, and the error bound
simultaneously by induction. Whenever $e_j\le2^{-20}$,
$\cos(2\arg z)\ge1-2(13/25)^2=287/625$ gives
\[
 \Rea h_{\ell,j}(z-\ell)
 \ge\frac{2^j}{|z|^2}
 \left(\frac{287}{625}\beta_{\ell,j}-\frac{b_\ell}{15}-e_j\right)
 \ge\frac{2^j}{|z|^2}\left(\frac{184}{625}-2^{-20}\right)>0.
\]
The Gamma representation and the bound for $e_0$ establish the initial
case. Suppose the assertions hold at layer $j<J$. Since $n/100>1$,
both unit shifts of $\mathcal D_{j+1}$ lie in $\mathcal D_j$, and
\eqref{eq:curvature-iteration} defines the next holomorphic curvature.
The segment between the normalized curvatures has real part greater
than $b_\ell/2$, and the corresponding $\phi$ arguments have modulus
less than $1/10$. With the chosen branches, the logarithmic identities
in \eqref{eq:Z-recurrence} hold on the real diameter and hence throughout
the disk. The gap $n/100$ and $|\phi'|<2$ therefore give
\begin{equation}\label{eq:physical-model-error}
 e_{j+1}\le(1+2^{16-j})e_j\le2^{17-j}e_j.
\end{equation}
Consequently $e_j<2^{-190+17j-j(j-1)/2}\le2^{-52}<2^{-20}$
for $j\le J\le12$. The same real-part estimate proves positivity
at the new layer and completes the induction.

\begin{proposition}\label{prop:small-tail}
For every integer $3\le\ell\le18$ and position $k$ with
$\ell+k\ge8192$, there is an integer $T=T(\ell,k)\ge0$ such that
\[
 a_v^{[j]}>0\quad
 (0\le j\le T,\ v\ge0,\ |v-k|\le T+1-j),
 \qquad c_T(k)\le q_*,
\]
where $a_v=d_\ell(\ell+v)$.
\end{proposition}
\begin{proof}
On $|z-n|\le2n/5$, put $B=2^J\beta_{\ell,J}>3072$.
Since $4/\sqrt{(3/5)8192}+2^{-20}/(3/4)<3/50$, we have
\[
 |z^2h_{\ell,J}(z-\ell)-B|<\frac3{50}B.
\]
Corollary~\ref{cor:normalized-stop} applies with $r=2/5$, because
$B(4/49-2/75)>128$. The real part of the degree variable exceeds $\ell$ on every required disk. By the agreement described in
Section~\ref{subsec:prepared-entry}, these curvatures are those of
the original iterates.
\end{proof}
\subsection{Finite initial ranges}\label{sec:certificates}

The remaining ranges are $1\le k\le8191-\ell$ for $3\le\ell\le18$
and $1\le k\le255$ for $\ell\ge19$. Below, we verify the inequalities
by rational enclosures on the larger ranges $k\le10000$ and
$k\le500$, respectively.
We use the arithmetic of
Subsection~\ref{sec:models}, retaining the same parameter in all neighboring
entries and all layers.

For $t=1/\lambda$, put
\[
 B_k(t)=\frac{k+1}{2\lambda}
 \frac{D_\lambda(\lambda+k+1)}{D_\lambda(\lambda+k)}.
\]
The initial ratios $c_0(k;t)$ of the sequence
$(D_\lambda(\lambda+k))_{k\ge0}$ are
\begin{equation}\label{eq:initial-ratio}
 c_0(k;t)=\frac k{k+1}\frac{B_k(t)}{B_{k-1}(t)},
 \qquad c_0(0;t)=0.
\end{equation}
From the moment formula and Lagrange inversion, we obtain
\[
 \sum_{k\ge0}\frac{D_\lambda(\lambda+k)}{D_\lambda(\lambda)}
 \frac{(\lambda+1)_k}{(2\lambda+1)_k}(4z)^k
 =(1-8z)^{-3/4}
 \left(\frac2{1+\sqrt{1-8z}}\right)^{2\lambda}.
\]
Differentiating this generating function and comparing coefficients yields
$B_0(t)=(1+t/2)(1+3t/2)/(1+t)$ and
\begin{align}
 B_{k+1}(t)&=U_k(t)-\frac{V_k(t)}{B_k(t)},\label{eq:Riccati}\\
 U_k(t)&=\frac{1+(4k+6)t+(2k^2+6k+19/4)t^2}{1+(k+2)t},\notag\\
 V_k(t)&=\frac{2(k+1)t[1+(k+1)t/2][1+(k+3/4)t][1+(k+5/4)t]}
 {[1+(k+1)t][1+(k+2)t]}.\notag
\end{align}
These expressions extend to $t=0$. On $a\le\lambda\le b$ we use
$t=a^{-1}+(b^{-1}-a^{-1})u$, $0\le u\le1$.
For $\lambda\ge2048$ we use $t=u/2048$.
To retain cancellation in \eqref{eq:Riccati}, we estimate the whole
residual. At each real parameter value, $|B-P|\le E$,
$P>E\ge0$, and $V\ge0$ imply
\[
 \left|U-\frac VB-Q\right|
 \le\left|U-\frac VP-Q\right|+\frac{VE}{P(P-E)}.
\]
We apply Lemma~\ref{lem:model-arithmetic} to the remaining products and quotients.
For a polynomial enclosure $(\sum_i p_i u^i,E)$ on $[0,1]$, the bounds
$p_0+\sum_{i\ge1}\min(p_i,0)-E$ and
$p_0+\sum_{i\ge1}\max(p_i,0)+E$ enclose its range.

Retaining the parameter is essential. For
$x_0=c_0(400;1/19)$, $x_-=c_0(399;1/23)$, and
$x_+=c_0(401;1/23)$, we use the integers $4^n d_\ell(n)$ from
\eqref{eq:coefficients}. For fixed $\ell$, the factor $4^n$ is
geometric in $n$ and leaves the ratios unchanged. We obtain
$x_0^2(1-x_-)(1-x_+)/(1-x_0)^2>36/25$.
Thus, using unrelated scalar bounds on $[19,23]$, one obtains an upper
bound exceeding one after the first update, at every precision. This failure reflects
the loss of parameter dependence.

\begin{proposition}\label{lem:finite-prefixes}
For the sixteen integer parameters $3\le\ell\le18$ and
$1\le k\le10000$, there is an integer $T=T(\ell,k)\ge0$ such that
\[
 a_v^{[j]}>0\quad
 (0\le j\le T,\ v\ge0,\ |v-k|\le T+1-j),
 \qquad c_T(k)\le q_*,
\]
where $a_v=d_\ell(\ell+v)$. For every real $\lambda$ in the
intervals of Table~\ref{tab:prefixes} and every integer $1\le k\le500$,
the same inequalities hold for $a_v=D_\lambda(\lambda+v)$
with some integer $T=T(\lambda,k)\ge0$.
\end{proposition}

\begin{proof}
We initialize by \eqref{eq:Riccati} and \eqref{eq:initial-ratio}, and
iterate \eqref{eq:ratio-iteration}. For a fixed parameter, the recurrence
is increasing in its central ratio and decreasing in its two neighboring
ratios on $[0,1)$. Hence rational endpoints with outward rounding suffice.
For a parameter interval, we retain $u$ in the polynomial enclosures above.

For the sixteen fixed values of $\ell$, the scale is $2^{1024}$ and the initial
ratios extend through $k=10033$. Discarding the artificial right endpoint
at each step, we can perform $32$ updates with every index
$k\le10000$ still retained.
The ratios at these indices satisfy the bound within $27$ iterations. For the parameter intervals, we use
the degrees and scales $2^b$ in Table~\ref{tab:prefixes}. Initial ratios
extend through $k=525$; at layer $j$, we retain $0\le k\le525-j$.
At the allowed depth $24$, every index $k\le500$ is therefore retained.
Only the original left endpoint is assigned the ratio $c_j(0)=0$.

At every computed layer, including the last, each retained
ratio has upper bound strictly less than one. Thus $1-c_j(k)$ has
a positive lower bound before every update. Nonnegativity follows
from the initial ratios and \eqref{eq:ratio-iteration}, even if an enclosure
has a negative lower bound. Thus the coefficient iterates are positive.
We continue updating every retained position after its ratio satisfies the bound. At each
index under consideration, the upper bound $q$ satisfies
\begin{equation}\label{eq:stop-rational}
 q<\frac12,\qquad q^2-3q+1>0,
\end{equation}
by the stated depth. These inequalities imply $q<q_*$.
All rounding and truncation errors are included by
Lemma~\ref{lem:model-arithmetic}.
\end{proof}

\begin{table}[ht]
\centering
\caption{Finite parameter ranges for $1\le k\le500$.}
\label{tab:prefixes}
\begin{tabular}{crrr}
\toprule
Parameter interval & Polynomial degree & Bits $b$ & Largest depth\\
\midrule
$[19,23]$&96&1024&14\\
$[24,32]$&96&1024&14\\
$[33,48]$&96&1024&14\\
$[49,80]$&128&1024&13\\
$[81,128]$&96&1024&12\\
$[129,256]$&96&1024&12\\
$[257,512]$&96&1024&11\\
$[512,2048]$&80&768&10\\
$[2048,\infty)$&48&768&10\\
\bottomrule
\end{tabular}
\end{table}
\FloatBarrier

The intervals cover every integer $\ell\ge19$. The assertion for real
parameters concerns their union and does not include the intervening gaps.

\begin{proof}[Proof of Theorem~\ref{thm:main}]
Fix \(\ell\ge3\). If $3\le\ell\le18$, we obtain positive local
iterates for all degrees $n\ge8192$ from Proposition~\ref{prop:small-tail}
and for all positions $k\le10000$ from Proposition~\ref{lem:finite-prefixes}.
These ranges overlap and cover the sequence. If \(\ell\ge19\), we use
Proposition~\ref{prop:uniform-tail} for \(k\ge256\) and the bounds in
Table~\ref{tab:prefixes} for \(1\le k\le500\). Thus, for each $k\ge1$,
there is a finite $T=T(k)$ such that $a_k^{[j]}>0$ for
$0\le j\le T$ and $c_T(k)\le q_*$. The initial coefficients are positive by
\eqref{eq:coefficients}, and the entry at the left endpoint is repeatedly squared.
The conclusion follows from Lemma~\ref{lem:local-stop} with $q=q_*$.
\end{proof}

\section{Euler's difference table}\label{sec:euler}
Euler's difference table arose in the study of derangements, that is,
permutations without fixed points \cite{Rakotondrajao}.
For $n\ge0$, its entries are defined by
\[
 E_{n,n}=n!,\qquad
 E_{n,k-1}=E_{n,k}-E_{n-1,k-1}\quad(1\le k\le n).
\]
Dumont and Randrianarivony \cite{DumontRandrianarivony} showed that
$E_{n,k}$ counts the permutations of $\{1,\ldots,n\}$ whose fixed
points all belong to $\{1,\ldots,k\}$. In particular, $E_{n,0}$ is
the number of derangements of $n$ elements, whereas $E_{n,n}=n!$
counts all permutations. Writing
$D_m=m!\sum_{s=0}^m(-1)^s/s!$ for the derangement numbers, we obtain
\[
 E_{n,k}=\sum_{s=0}^k\binom ksD_{n-s}\qquad(0\le k\le n)
\]
by choosing the set of fixed points and deranging the remaining elements.

We consider the normalized entries
\[
 \widehat E_{n,k}=\frac{E_{n,k}}{k!}\qquad(0\le k\le n).
\]
These are integers with a further combinatorial interpretation.
Rakotondrajao \cite{Rakotondrajao} showed that $\widehat E_{n,k}$ counts
permutations with no fixed points among $k+1,\ldots,n$ and with
the elements $1,\ldots,k$ lying in distinct cycles.

Chen, Gu, Ma, and Wang \cite{EulerOriginal} proved that, for every
$n\ge1$, the normalized sequence
$(\widehat E_{n,0},\ldots,\widehat E_{n,n})$ is log-concave and
essentially $2$-log-concave. They conjectured that these sequences
are essentially infinitely log-concave.
Here essential iteration means that both endpoints are discarded
after each application of $\LL$. More precisely, set
\[
 E_{n,k}^{[0]}=\widehat E_{n,k}\qquad(0\le k\le n).
\]
For $0\le r<\lfloor n/2\rfloor$, define
\[
 E_{n,k}^{[r+1]}=(E_{n,k}^{[r]})^2
 -E_{n,k-1}^{[r]}E_{n,k+1}^{[r]}
 \qquad(r+1\le k\le n-r-1).
\]
Thus the indices at depth $r$ are $r\le k\le n-r$.
The ratios $c_r(k)$ and curvatures $h_r(k)$ refer to these normalized
iterates, as in Section~\ref{sec:stopping}.
We prove the conjecture, with strict positivity for $n\ge3$.
\begin{theorem}\label{thm:euler}
For every integer $n\ge3$, the essential iterates of
$(\widehat E_{n,0},\ldots,\widehat E_{n,n})$ satisfy
\[
 E_{n,k}^{[r]}>0\qquad
 (1\le r\le\lfloor n/2\rfloor,\ r\le k\le n-r).
\]
\end{theorem}
For $n=0,1$, there are no essential iterates of positive order.  For $n=2$ the sequence
is $(1,1,1)$ and its sole first-order interior entry is zero.

The essential convention is necessary. For $n=3$, the normalized sequence is
$(\widehat E_{3,0},\ldots,\widehat E_{3,3})=(2,3,2,1)$.
If this sequence is instead extended by zero, its second iterate at index
$2$ equals $-4$.

The factor $1/k!$ contributes a curvature of order $k^{-1}$.
To estimate the remaining factor, we interpolate $E_{n,k}$ and compare
the interpolation with an exponential. The logarithm of that
exponential is affine and disappears under the second difference.
Define
\begin{equation}\label{eq:euler-interpolation}
 \mathscr E_n(z)=\sum_{s=0}^n(-1)^s\binom{n-z}{s}(n-s)!.
\end{equation}
We apply inclusion-exclusion to the possible fixed points among
$k+1,\ldots,n$. Since $E_{n,k}\ge D_n>0$ for $n\ge3$, we obtain
\[
 \mathscr E_n(k)=E_{n,k}>0\qquad(0\le k\le n,\ n\ge3).
\]

For $n\ge750$ and $|1-w|\le1$, we have
\begin{equation}\label{eq:euler-exponential-approx}
 |\mathscr E_n(nw)/n!-e^{w-1}|<1/100.
\end{equation}
Indeed, put $v=1-w$ and
$u_s=(v-s/n)/(1-s/n)$, $P_r=\prod_{s<r}u_s$.  Then
$\mathscr E_n(nw)/n!=\sum_{r=0}^n(-1)^rP_r/r!$.
For $s\le19$, we have $|u_s|\le9/8$ and $|u_s-v|\le3s/n$.
Applying these bounds to a telescoping expansion of $P_r-v^r$ gives
\[
 \sum_{r=0}^{20}\frac{|P_r-v^r|}{r!}
 \le\frac{27e^{9/8}}{16n}<\frac6n.
\]
For all $r\le n$, $|P_r|\le6^r$, since
$n(n-1)\cdots(n-r+1)\ge(n/e)^r$.  The two omitted exponential tails
are bounded by $3\cdot6^{21}/21!$.  Thus
$6/n+3\cdot6^{21}/21!<1/100$, proving
\eqref{eq:euler-exponential-approx}.

Fix integers $n\ge750$ and $128\le k\le n$. Define
\[
 L_n(z)=\log\{e^{1-z/n}\mathscr E_n(z)/n!\},\qquad
 h_0(z)=\log(1+1/z)-\D L_n(z).
\]
On $|z-k|\le3k/4$, we have $|1-z/n|\le1-k/(4n)<1$.
Thus the logarithm series and \eqref{eq:euler-exponential-approx}
give $|L_n|<1/30$ there. On the disk of radius $3k/5$,
\eqref{eq:modulus-difference} bounds $|\D L_n|<3/k^2$, while the
logarithm series gives $|\log(1+1/z)-z^{-1}|<4/k^2$. Thus
\begin{equation}\label{eq:euler-seed}
 |h_0(z)-z^{-1}|<7/k^2\qquad(|z-k|\le3k/5).
\end{equation}
At the center, using the full radius $3k/4$ instead gives
$|\D L_n(k)|<1/(8k^2)$, and consequently
\begin{equation}\label{eq:euler-center-seed}
 k^{-1}-k^{-2}<h_0(k)<k^{-1}+k^{-2}.
\end{equation}

\begin{proposition}\label{lem:euler-analytic}
Let $n\ge750$, $128\le k\le n$, and $M=\lceil\log_2k\rceil+1$.
Put $R_j=(3k/5)(3/4)^j$ for $0\le j<M$ and $R_M=0$.
The curvature iterates $h_j$ are holomorphic near $|z-k|\le R_j$
and have positive real parts there, with $h_M(k)>19/16>\log3$.
\end{proposition}
\begin{proof}
The successive gaps among $R_0,\ldots,R_{M-1}$ have square greater than six. For $M=8$ use
$k\ge128$; for $M\ge9$ use $k>2^{M-2}$, with the smallest bound at
$M=9$. The latter squared gap is greater than
$((3\cdot128/20)(3/4)^7)^2>6$.
Corollary~\ref{cor:geometric-budget} and the first term of
\eqref{eq:mobius-error} give, for $1\le j<M$,
\[
 \widetilde{\mathcal E}_j<\frac{87}{k^2},\qquad
 k^2\widetilde{\mathcal E}_1
 =\frac{8/7}{9/400-k^{-2}}<51.
\]
On a disk of radius $rk$, \eqref{eq:euler-seed} gives
$k^2\Rea h_0\ge k/(1+r)-7$. At $r=3/5,9/20,27/80$ these lower
bounds are at least $73$, greater than $81$, and greater than $88$,
respectively. Comparing with the preceding error bounds, we obtain
$\inf_{|z-k|\le R_j}\Rea h_0>\widetilde{\mathcal E}_j$
for $0\le j<M$.
The radius $R_{M-1}$ is three times the smallest of these gaps,
so $R_{M-1}^2>54$. By \eqref{eq:geometric-center-error},
\[
 \mathcal C_M<\frac{18}{(1-1/54)(3/5)^2k^2}<\frac{51}{k^2}.
\]
Since $h_0(k)>\mathcal C_M$, Theorem~\ref{thm:unified-entry}
with \eqref{eq:mobius-error} applies also at the center. Using
\eqref{eq:euler-center-seed} and $2^M\ge2k$, we obtain
\[
 h_M(k)>\frac{2^M}{k}\left(1-\frac{52}{k}\right)
 \ge\frac{19}{16}>\log3.
\]
For $1\le j\le M$, $R_{j-1}>M+1-j$.
At integer arguments required by essential iteration, $h_0$
agrees with the original sequence curvature. The corresponding
initial entries and both unit shifts are therefore included wherever
they are used.
\end{proof}
If $M<\min\{k,n-k\}$, then $T_{1/3}(k)\le M$. By \eqref{eq:center-error}, we also have $h_j(k)<2^j(k^{-1}+52k^{-2})$
for $0\le j\le M$, which gives
$T_{1/3}(k)=\log_2k+O(1)$ uniformly in $n$ in this range. Otherwise, the same estimates prove positivity through the last
essential coefficient layer. We identify the analytic and sequence
recurrences only on these admissible index ranges.

For the remaining initial indices, we verify the inequalities using
exact rational bounds on the shrinking support of essential iteration.

\begin{proposition}\label{ver:euler}
For every integer $n\ge1500$,
\[
 E_{n,k}^{[r]}>0\quad(0\le r\le11,\ r\le k\le1500-r),
 \qquad c_{11}(k)\le3/10<1/3\quad(12\le k\le1488).
\]
These quantities depend only on the original initial entries
$(\widehat E_{n,0},\ldots,\widehat E_{n,1500})$.
The ratio inequality covers all $1477$ retained indices.
\end{proposition}
\begin{proof}
Put $x=1500/n\in(0,1]$ and $F_k=e\mathscr E_n(k)/n!$.
From the derangement formula and $|eD_q-q!|<3$, including $q=0,1$,
we obtain
\[
 \left|F_k-\sum_{s=0}^k\frac{\binom{k}{s}}{(n)_{\underline s}}\right|
 \le\frac{3\cdot2^{1500}}{1500!}<2^{-4096}
 \qquad(n\ge1500,\ 0\le k\le1500).
\]
Let $\mathrm S(a,b)$ denote the Stirling numbers of the second kind.
From their generating function, we obtain the degree-$32$ Taylor polynomial
\[
 P_k(x)=1+\sum_{m=1}^{32}\frac{x^m}{1500^m}
              \sum_{s=1}^{\min(k,m)}\binom{k}{s}\mathrm S(m-1,s-1).
\]
All Taylor coefficients are nonnegative. The discarded part from
$s\le32$ is therefore bounded by its value at $x=1$, while each term
with $s>32$ is at most $1/s!$. Consequently
\[
 |F_k-P_k(x)|\le
 \sum_{s=0}^{\min(k,32)}\frac{\binom{k}{s}}{(1500)_{\underline s}}
 -P_k(1)+\eta_k+2^{-4096},
\]
where $\eta_k=0$ for $k\le32$ and $\eta_k=2/33!$ otherwise.
We thus obtain a rational polynomial enclosure uniformly for every
integer $n\ge1500$.

We apply Lemma~\ref{lem:model-arithmetic} with the same variable $x$
in every ratio, rounding coefficients downward and error bounds upward
at scale $S=2^{384}$.
The initial ratios are
\[
 c_0(k)=\frac{k}{k+1}\frac{F_{k-1}F_{k+1}}{F_k^2},
 \qquad 1\le k\le1499.
\]
The initial denominator satisfies $F_k^2>(S-2)/S$, since
$F_k\ge1-2^{-4096}$. Before each subsequent update, the bound
$U/S<1$ implies that the squared denominator is at least $(1-U/S)^2$.
The recurrence then preserves nonnegativity of the ratios.
After $r$ updates the retained indices are
$r+1\le k\le1499-r$, with no imposed boundary values.
These ranges contain all predecessors needed for essential iteration.
At $r=11$ all $1477$ upper bounds are below $297/1000<3/10$.
\end{proof}

\begin{proof}[Proof of Theorem~\ref{thm:euler}]
We first consider $3\le n\le1500$. We compute the initial ratios from
the exact integer entries $\widehat E_{n,k}$ and enclose them by rational
intervals with outward rounding at scale $2^{128}$.
Both endpoints are discarded after each iteration. At each position,
the ratio is at most $1/3$ within eleven iterations, unless the last
essential coefficient layer occurs earlier. All entries remain positive
through these iterations.

For $n>1500$, define
\[
 T(k)=\begin{cases}
 0,&k=0,\\
 k,&1\le k\le11,\\
 11,&12\le k\le127,\\
 \min\{\lceil\log_2k\rceil+1,k,n-k\},&128\le k\le n.
 \end{cases}
\]
For $1\le k<128$, Proposition~\ref{ver:euler} gives positivity
through layer $T(k)$ and $c_{T(k)}(k)\le3/10<1/3$ whenever
$T(k)<\min\{k,n-k\}$. For $k\ge128$, Proposition~\ref{lem:euler-analytic}
gives positivity through layer $T(k)$ and $c_{T(k)}(k)<1/3$ whenever
$T(k)<\min\{k,n-k\}$, since at interior integer arguments
\[
 c_0(t)=\frac{t}{t+1}
 \frac{\mathscr E_n(t-1)\mathscr E_n(t+1)}{\mathscr E_n(t)^2}
 =e^{-h_0(t)}\qquad(1\le t<n).
\]
At $k=0,n$, we have $T(k)=0=\min\{k,n-k\}$, so only initial
positivity is required. Thus, for every $n\ge3$, the hypotheses of
Lemma~\ref{lem:local-stop} hold at every position with $q=1/3$.
The lemma gives positivity at every essential iteration depth.
\end{proof}
\section{The Medina--Moll--Rowland polynomials}\label{sec:mmr}
Medina, Moll, and Rowland \cite[Theorem~2.1]{MMROriginal}
studied the repeated integration of $\log(1+x)$. For $x>-1$, define
\[
 f_0(x)=\log(1+x),\qquad
 f_n(x)=\int_0^x f_{n-1}(t)\,dt\quad(n\ge1).
\]
They obtained
\[
 f_n(x)=\frac{(1+x)^n}{n!}\log(1+x)-xA_n(x),
\]
where $A_n(x)$ is a polynomial of degree $n-1$ for $n\ge1$.
Writing $H_n=\sum_{s=1}^n1/s$ and $H_0=0$, their formula is
\begin{equation}\label{eq:mmr-polynomials}
 A_n(x)=\frac1{n!}\sum_{s=1}^n\binom ns(H_n-H_{n-s})x^{s-1}.
\end{equation}
Medina, Moll, and Rowland
\cite[Theorem~4.1, Conjecture~4.3, and Theorem~4.5]{MMROriginal}
proved that these coefficient sequences are log-concave for every
$n\ge1$, conjectured infinite log-concavity, and verified the
conjecture for $n\le300$.
We prove the conjecture for every $n\ge1$.
\begin{theorem}\label{thm:mmr}
For every $n\ge1$, the coefficient sequence of $A_n$ is strictly infinitely
log-concave on its original support.  Its $j$th iterate is $8/3$-factor
log-concave whenever $j\ge\max\{8,\lfloor\log_2n\rfloor+1\}$.
\end{theorem}
The polynomials $A_n$ need not have only real zeros. For example,
$36A_3(x)=11x^2+15x+6$ has discriminant $-39$.
We therefore cannot apply the real-rootedness criterion of
Br\"and\'en \cite{Branden} directly to this family.

Put $a_k=\binom nk(H_n-H_k)$ for $0\le k\le n-1$.
These are the reversed coefficients of $n!A_n$. Reversal and positive
rescaling preserve positivity and coefficient ratios. We separate the
binomial factor from a correction with positive real part. The series
\begin{equation}\label{eq:mmr-stieltjes}
 M_n(z)=\sum_{s=0}^{\infty}\frac1{(s+n+1)(s+z+1)},\qquad\Rea z>-1,
\end{equation}
converges locally uniformly and has positive real part.  Partial
fractions give
\[
 M_n(k)=\frac{H_n-H_k}{n-k},\qquad
 a_k=n\binom{n-1}{k}M_n(k).
\]
Its principal logarithm has imaginary part in $(-\pi/2,\pi/2)$.
Consequently the initial curvature is interpolated by
\begin{equation}\label{eq:mmr-h0}
 h_0(z)=\log(1+1/z)+\log(1+1/(n-1-z))-\D\log M_n(z).
\end{equation}

\begin{proposition}\label{prop:mmr-interior}
For $d=\min(k,n-1-k)\ge128$, put $L=\lfloor\log_2d\rfloor+2$.
Then
\[
 a_v^{[j]}>0\quad(0\le j\le L,\ 0\le v\le n-1,\ |v-k|\le L+1-j),
 \qquad h_L(k)>87/64>\log(8/3).
\]
In particular, $c_L(k)<3/8$.
\end{proposition}
\begin{proof}
Set $R_j=(4d/5)(3/4)^j$. On $|z-k|\le rd$, we estimate the nearer
endpoint term in \eqref{eq:mmr-h0} and apply \eqref{eq:halfplane-strip}
to obtain
\[
 d^2\Rea h_0(z)\ge
 \frac{d^2}{(1+r)d+1}-\frac{2}{(1-r)^2}.
\]
For $r=4/5,3/5,9/20$, respectively, the right side exceeds $20,67,81$
at $d=128$ and increases with $d$. These bounds apply to the initial
disk, the first smaller disk, and all subsequent disks.
Since $L\ge9$ and $d\ge2^{L-2}$, the smallest radius gap is at least
$(3/2)^L/15$, whose square exceeds $6$. Hence, by
Corollary~\ref{cor:geometric-budget},
\[
 \mathcal E_j<\frac{315}{4d^2}<\frac{79}{d^2},\qquad
 d^2\mathcal E_1=\frac{2}{1/25-d^{-2}}<51.
\]
These bounds verify the hypotheses on the real part in
Theorem~\ref{thm:unified-entry}.
At the center, $h_0(k)>d^{-1}-3d^{-2}$ also holds by
\eqref{eq:halfplane-strip}. Consequently
\[
 h_L(k)>\frac{2^L(d-82)}{d^2}\ge\frac{87}{64}.
\]
For the last inequality, on $2^{L-2}\le d\le2^{L-1}$ the minimum occurs
at an endpoint; both endpoint bounds increase with $L\ge9$.
Since $R_0+1<d$ and every radius gap exceeds one, all indices required
by the recurrence remain within the original support.
\end{proof}
The same argument also yields
$h_0(k)\le2/d+2/d^2$ and $h_j(k)<3\cdot2^j/d$ through $L$. Hence
\begin{equation}\label{eq:mmr-stopping}
 \left\lceil\log_2\frac{\rho_*d}{3}\right\rceil
 \le T_{q_*}(k)\le\lfloor\log_2d\rfloor+2,
 \qquad d\ge128.
\end{equation}

It remains to consider distances at most $128$ from either endpoint.
At these positions, we express the ratios as rational functions of
two independent parameters
\[
 0\le w\le1/1024,\qquad 0\le t\le2/15.
\]
Put $v_k=t/(1-tH_k)$ for $1\le k\le136$ and
$f_k(w)=(k+1)^{-1}\sum_{s=0}^k(1-sw)^{-1}$ for $0\le k\le137$.
For $1\le k\le136$, the families are
\begin{align}
 c^L_0(k;w,t)&=\frac{k}{k+1}\frac{1-kw}{1-(k-1)w}
 \left(1+\frac{v_k-v_k^2}{k(k+1)}\right),\label{eq:mmr-left}\\
 c^R_0(k;w)&=\frac{k}{k+1}\frac{1-(k+1)w}{1-kw}
 \frac{f_{k-1}(w)f_{k+1}(w)}{f_k(w)^2}.\label{eq:mmr-right}
\end{align}
Set $c^L_0(0;w,t)=c^R_0(0;w)=0$.
For the sequence $(a_k)_{k=0}^{n-1}$ defined above, the initial ratios satisfy
\[
 \begin{aligned}
 c^L_0(k;1/n,1/H_n)&=c_0(k),\\
 c^R_0(k;1/n)&=c_0(n-1-k)
 \end{aligned}
 \qquad(0\le k\le136,\ n\ge1024).
\]
Here left and right refer to the order of $a_k$, which is the reverse
of the coefficient order of $A_n$.
By the elementary rational
bounds $H_{1024}>15/2$ and $H_{137}<6$, the parameters for every
$n\ge1024$ lie in the rectangle, and every initial denominator is positive.

After removing positive geometric factors, the left entries are
\[
 \frac{1-tH_k}{k!}\prod_{s=0}^{k-1}(1-sw),
\]
while $a_{n-1-k}/n^k=f_k(w)\prod_{s=1}^k(1-sw)/k!$ on the right.
Thus \eqref{eq:mmr-left}--\eqref{eq:mmr-right} hold, and the initial
entries through distance $137$ are positive. The parameters are
kept independent throughout the finite boundary calculation;
no approximation to $H_n$ is needed.

\begin{proposition}\label{ver:mmr}
For $n\ge1024$,
\[
\begin{aligned}
 &a_v^{[j]}>0\quad
 (0\le j\le8,\ 0\le v\le n-1,\ \min(v,n-1-v)\le137-j),\\
 &c_8(k)<3/20\quad(0\le k\le n-1,\ \min(k,n-1-k)\le128).
\end{aligned}
\]
For $1\le n\le1023$, there is an integer $0\le T\le8$ such that
$a_k^{[j]}>0$ for $0\le j\le T$, $0\le k\le n-1$, and
$c_T(k)\le3/8$ for $0\le k\le n-1$.
\end{proposition}
\begin{proof}[Computer-assisted proof]
For the left family use the eight rectangles
\[
 w=\frac{1+x}{2048},\qquad t=\frac{2i+1+y}{120},
 \qquad |x|,|y|\le1,\quad 0\le i\le7.
\]
For the right family, only the parameter $x$ is needed. We use the norm
$\|P\|_1=\sum_{a,b}|[x^ay^b]P|$, which is submultiplicative and bounds
$|P|$ on these rectangles. Thus Lemma~\ref{lem:model-arithmetic}
applies without change. We use total degree $24$ and dyadic scale $2^{256}$,
retaining both variables through every operation and bounding every
rounding error and discarded coefficient. To choose a quotient
approximation to $P/Q$, let $[F]_{m,S}$ denote truncation at total
degree $m$, with each retained coefficient rounded downward to a
multiple of $S^{-1}$. Starting from
$B=S^{-1}\lfloor S/Q(0,0)\rfloor$, apply
\[
 B\leftarrow[B(2-[QB]_{m,S})]_{m,S}
 \qquad\text{for }m=1,3,7,15,24,
\]
and set $A=[PB]_{24,S}$. We then use the full residual $P-QA$ in
Lemma~\ref{lem:model-arithmetic} to certify the quotient. This
certification does not require convergence of the Newton iterates.

Starting with $0\le k\le136$, we apply \eqref{eq:ratio-iteration} eight times, keeping the original endpoint ratio zero and
retaining $0\le k\le136-j$ at layer $j$. Before every update, we verify
a positive denominator bound and an upper ratio below one by exact
integer comparisons. After eight iterations, all upper ratio bounds are
less than $7/50<3/20$. The initial ratios are nonnegative throughout the rectangle, since
$0\le v_k\le2/3$. By induction, the corresponding iterates of the
original sequence are positive. All required initial coefficient
indices lie within distance $137$.

For $1\le n\le1023$, we compute rational interval bounds for the ratios from
the coefficients in \eqref{eq:mmr-polynomials}, with outward rounding at scale $2^{128}$.
We retain both original endpoints, with ratios zero, and apply the same
directed updates. Every preceding upper ratio is below one, and all
ratios are at most $3/8$ within eight updates. The cases $n=1,2$ need no update.
\end{proof}

\begin{proof}[Proof of Theorem~\ref{thm:mmr}]
The coefficients are positive, and the entries at the original endpoints
are repeatedly squared. For $n\le1023$, we apply Proposition~\ref{ver:mmr}.
For $n\ge1024$, take $T(k)=8$ when $\min(k,n-1-k)\le128$, by the
same proposition. At the remaining positions, take
$T(k)=\lfloor\log_2d\rfloor+2\le\lfloor\log_2 n\rfloor+1$, by
Proposition~\ref{prop:mmr-interior}. In both cases,
$a_k^{[j]}>0$ for $0\le j\le T(k)$ and $c_{T(k)}(k)\le3/8$.
With $q=3/8$, Lemma~\ref{lem:local-stop} implies positivity at all
orders and preservation of this ratio bound. Reversal and
multiplication by $n!$ preserve the conclusion.
\end{proof}

\section{Power sequences}\label{sec:powers}
For real $d>0$, put $a^{[0]}_k(d)=k^d$ for $k\ge0$, and use zero
extension to negative indices. Thus $a^{[j]}_0=0$ and $a^{[j]}_1=1$
for every $j\ge0$. The positive support starts at one.
Br\"and\'en and Chasse
\cite[Question~1]{BrandenChasse} asked whether every positive integer
exponent gives an infinitely log-concave sequence.

Hou and Li \cite[Theorem~3.2 and Example~3.1]{HouLi} proved that,
for every real $d\ge2$ and every fixed positive integer $r$,
the sequence $(k^d)_{k\ge0}$ is asymptotically $r$-log-concave.
The starting index from which the first $r$ iterates are nonnegative
may depend on $d$ and $r$.
This includes $d=2$, although the original sequence is not infinitely
log-concave, as shown below.
We establish the following result for all positive indices and all iteration orders.
\begin{theorem}\label{thm:powers}
For every real $d\ge3$, all iterates of $(k^d)_{k\ge0}$ are strictly
positive at every index $k\ge1$.
For a positive integer $d$, this sequence is infinitely log-concave if and
only if $d\ne2$.
\end{theorem}
For $d\ge3$, we also obtain the following iteration bounds.
\begin{corollary}\label{cor:power-entry-bounds}
For the sequence $a_k=k^d$, $k\ge0$, with real $d\ge3$,
we have $T_{3/8}(k)\le17$ for $2\le k\le256$, and
\[
 T_{3/8}(k)\le\max\left\{0,
 \left\lceil\log_2\frac{3k^2}{d-2}\right\rceil\right\}
 \qquad(k\ge256).
\]
\end{corollary}

\subsection{Analytic estimates}
The initial curvature has a simple convergent expansion. Write
\[
 v(z)=-\Log(1-z^{-2})=\sum_{r\ge1}\frac{z^{-2r}}r,
 \qquad |z|>1.
\]
For the power sequence, $h_0(z)=d\,v(z)$. The coefficients of
$z^{-2}$ in the subsequent curvature expansions are
\begin{equation}\label{eq:power-leading}
 p_j(d)=2+(d-2)2^j.
\end{equation}
For $d\ge241$, we estimate the successive curvatures starting from
the first iterate.
For $3\le d\le242$, we prove the required inequalities analytically
for $k\ge256$ and verify them by finite computation for $1\le k\le256$.

For large exponents, we obtain a sharper curvature bound by factoring
the first iterate. The additional factor is log-concave, as shown below.

\begin{lemma}\label{lem:power-first}
For real $d>1$ define
$U_d(x)=(1-(1-x^{-2})^d)/(dx^{-2})$, $x\ge1$.
Then $\log U_d$ is strictly concave on $(1,\infty)$ and continuous at
one. The first two iterates of $(k^d)_{k\ge0}$ are positive at all
positive indices, and
\begin{equation}\label{eq:power-first-ratio}
 c_1(k)\le(1-k^{-2})^{2d-2}<1\qquad(k\ge2).
\end{equation}
\end{lemma}
\begin{proof}
The first iterate factors into a power and the function $U_d$.
It is enough to prove that $\log U_d(x)$ is strictly concave. Put
$y=1-x^{-2}$ and
\[
 L(y)=\log\frac{1-y^d}{d(1-y)}
 =\log\int_0^1(y+(1-y)t)^{d-1}\,dt.
\]
Then $L'>0$ and
$(1-y)L'=1-d(1-y)y^{d-1}/(1-y^d)$ is nonincreasing, because
\[
 \frac{1-y^d}{d(1-y)y^{d-1}}
 =\int_0^1(1+(y^{-1}-1)t)^{d-1}\,dt
\]
is decreasing.  Hence $(1-y)L''\le L'$ and
$(\log U_d)''=4x^{-6}L''-6x^{-4}L'\le-2x^{-4}L'<0$.
The value at one is $U_d(1)=1/d$.  The exact factorization
\begin{equation*}
 a_k^{[1]}=k^{2d}-(k^2-1)^d=dk^{2d-2}U_d(k)>0
\end{equation*}
gives \eqref{eq:power-first-ratio} for every $k\ge2$, with the case
$k=2$ following by continuity of $U_d$ at one.
At $k=1$, zero extension gives $c_1(1)=0$, and hence the second
iterate is positive at every positive index.
\end{proof}

\begin{proposition}
\label{prop:large-real-powers}
For every real $d\ge241$, $a_k^{[j]}(d)>0$ for all $j\ge0$
and $k\ge1$. For each $k\ge2$, there is an integer $T=T(d,k)\ge0$
such that $c_T(k)<3/8$.
\end{proposition}
\begin{proof}
Put $P=2d-2\ge480$. By Lemma~\ref{lem:power-first}, the first two
iterates are positive and $h_1(k)\ge Pv(k)\ge P/k^2$. For
$2\le k\le\sqrt P$, we therefore have $h_1(k)\ge1$ and $c_1(k)<3/8$.
The endpoint at one has ratio zero.

Fix an integer $k>\sqrt P$, so $k\ge22$. With the branch of $\phi$
analytic at zero, put $G(z)=\phi(dv(z))-\phi(v(z))$.
Differentiating the product for $\sinh$ \cite[Eq.~4.36.1]{DLMF} and using
$\sum m^{-2}=\pi^2/6$ and $\pi>3$ gives, for $|w|\le5/2$,
\[
 |\phi'(w)|\le\frac12+\frac{|w|}{12(1-|w|^2/36)}<\frac{19}{25},
 \qquad
 |\phi''(w)|\le\frac{1+|w|^2/36}{12(1-|w|^2/36)^2}<\frac3{20}.
\]
On $|z-k|\le k/2+1$, we have $|z|\ge5k/11\ge10$ and
$|dv(z)|<(k^2+2)/(2((5k/11)^2-1))\le27/11<5/2$.
Thus $G$ is holomorphic there. Since $1-e^{-v}=z^{-2}$,
$e^G=U_d$; on the positive ray $G=\log U_d$, so
$h_1=Pv-\D G$ on $|z-k|\le k/2$ agrees with the sequence curvature.

Direct differentiation gives $|v'|<203/(100|z|^3)$ and
$|v''|<123/(20|z|^4)$ when $|z|\ge10$. Combining these bounds with
$(d+1)/k^2<51/100$ and $(d^2+1)/k^4<253/1000$ in the chain rule gives
\[
 k^2|G''(z)|<\frac4{25s^6}+\frac{12}{5s^4}
 \qquad(|z|\ge sk,\ s\ge5/11).
\]
For $|z-k|\le rk$, use $s=1-r-1/22$ on the unit shifted segments
and integrate $G''$ against $1-|t|$ on $[-1,1]$. Also,
\[
 P|v(z)-z^{-2}|<
 \frac{1}{2k^2(1-r)^4\{1-1/(484(1-r)^2)\}}.
\]
Substitution of $r=1/2$ and $r=3/8$ yields
\[
 |h_1(z)-Pz^{-2}|<\frac1{k^2}
 \begin{cases}
 84,&|z-k|\le k/2,\\
 29,&|z-k|\le3k/8.
 \end{cases}
\]

Set $K=\lceil\log_2(2k^2/P)\rceil\ge2$ and
$R_j=(k/2)(3/4)^j$. Since $k^2>P2^{K-2}$, the smallest squared
radius gap exceeds $(P/128)(9/8)^{K-1}\ge9P/1024>4$.
By Corollary~\ref{cor:geometric-budget}, $\mathcal E_j<224/k^2$.
On the outer disk, $\Rea h_1>(4P/9-84)/k^2>0$.
Every subsequent disk has radius at most $3k/8$, and hence
\[
 \Rea h_1>\frac{64P/121-29}{k^2}>\frac{224}{k^2}>
 \mathcal E_j,
 \qquad \frac{64\cdot480}{121}-29-224=\frac{107}{121}>0.
\]
Since $\mathcal E_0=0$, Theorem~\ref{thm:unified-entry} applies.
Using the real bound $h_1(k)\ge P/k^2$, we obtain
\[
 h_{1+K}(k)>\frac{2^K(P-224)}{k^2}
 \ge2\left(1-\frac{224}{P}\right)
 \ge\frac{16}{15}>1>\log(8/3).
\]
All radius gaps exceed two. The unit shifts stay within the preceding
disks, and the initial indices needed for the first iterate lie in the
positive support. Hence the holomorphic functions agree with the
curvatures of the original sequence on the required integer ranges.
Together with positivity of the first two iterates, this gives
$a_k^{[j]}(d)>0$ for $0\le j\le1+K$ and $c_{1+K}(k)<3/8$.
Lemma~\ref{lem:local-stop} now applies.
\end{proof}

For bounded exponents, several preliminary iterations are needed.
We estimate the complete analytic remainder while keeping the cancellation
in the linear part of the normalized recurrence. For these estimates, $\pnorm{f}{R}=\sum_m|f_m|R^m$ for
$f(w)=\sum_m f_mw^m$.

\begin{lemma}\label{lem:power-scaled-difference}
For $p\ge4$, put $R=(4\sqrt{p-2})^{-1}$, $r=R/\sqrt2$, and
\[
 \mathscr Sf(w)=\frac{f(w/(1-w))-2f(w)+f(w/(1+w))}{w^2},
 \qquad \mathscr A_pf=\frac{2pf-\mathscr Sf}{2p-2}.
\]
If $f$ is even, $f(0)=0$, and $\pnorm{f}{R}<\infty$, then
\[
 \pnorm{\mathscr Sf}{r}\le6\pnorm{f}{R},\qquad
 \pnorm{\mathscr A_pf}{r}\le a(p)\pnorm{f}{R},\qquad
 a(p)=\max\left\{\frac12,\frac{24}{5(2p-2)}\right\}.
\]
In particular, $a(p)\le4/5$, and $a(p)=1/2$ for $p\ge6$.
\end{lemma}
\begin{proof}
For even $m\ge2$, put $\nu=m(m+1)$ and
\[
 S_m(r)=\frac{(1-r)^{-m}+(1+r)^{-m}-2}{r^2}
       =2\sum_{\ell\ge1}\binom{m+2\ell-1}{2\ell}r^{2\ell-2}.
\]
The monomial norms, divided by $R^m$, are
$2^{-m/2}S_m(r)$ and
$2^{-m/2}(|2p-\nu|+S_m(r)-\nu)/(2p-2)$, respectively.
For $\nu\le2p$, we obtain from Taylor's remainder formula
\[
 \frac{S_m(r)-\nu}{\nu}
 \le\frac{(m+2)(m+3)r^2}{12(1-r)^{m+4}}<\frac12.
\]
Indeed, $(m+2)(m+3)r^2\le1/3$ and $(m+5)r\le1$, using
$r^2\le\min\{1/64,1/[16(\nu-4)]\}$; hence
$(1-r)^{-m-4}\le e<3$. Hence
$\pnorm{\mathscr A_p(w^m)}{r}/R^m\le1/2$.

For the other bounds, $r\le1/8$ and positivity of the displayed series
reduce the calculation to $p=4$. Set
$x=32/49$, $y=32/81$, and $m=2k$. The two relevant expressions are
\[
 A_k=64(x^k+y^k-2^{1-k}),\qquad
 B_k=\frac{A_k-8\,2^{-k}}6\quad(k\ge2).
\]
Both have their maximum at $k=3$ (for $A_k$, take $k\ge1$).
To see this, let $g_k=(1-x)(2x)^k+(1-y)(2y)^k$.
It increases for $k\ge1$, since
\[
 g_{k+1}-g_k=\frac{255}{2401}(64/49)^k-
                    \frac{833}{6561}(64/81)^k>0,
 \qquad g_2<1<17/16<g_3.
\]
The differences of $A_k$ and $B_k$ have the signs of $1-g_k$ and
$17/16-g_k$. The single rational comparison
\[
 A_3=\frac{360864356336}{62523502209}<\frac{29}{5}
\]
gives $A_k<6$ and $B_k<4/5$. If $\nu>2p$, then
$\pnorm{\mathscr A_p(w^m)}{r}/R^m\le6B_k/(2p-2)\le24/[5(2p-2)]$. Absolute summation proves the lemma.
The substitutions are legitimate because $r/(1-r)<R$.
\end{proof}

\begin{proposition}\label{prop:power-analytic-preparation}
For every real $d\ge3$ and $j\ge1$, put
\[
 p_j=2+2^j(d-2),\qquad R_j=(4\sqrt{p_j-2})^{-1}.
\]
The iterated curvature has the holomorphic representation
\[
 h_j(z)=\frac{p_j}{z^2}\{1+e_j(1/z)\},\qquad
 \pnorm{e_j}{R_j}<\min\left\{\frac18,\frac{j}{4(p_j-2)}\right\},
\]
where $e_j$ is even and holomorphic near $|w|\le R_j$, with $e_j(0)=0$.
Consequently,
\begin{equation}\label{eq:power-analytic-error}
 \left|h_j(z)-\frac{p_j}{z^2}\right|
 <\frac{4jp_j}{|z|^4}\qquad(|z|\ge4\sqrt{p_j-2}).
\end{equation}
At positive integer arguments in these ranges, the curvatures agree
with those of the original sequence and are positive.
\end{proposition}
\begin{proof}
Write $V(w)=-\log(1-w^2)/w^2-1$ and
$\Psi(t)=t^2/[2(1-t)]$. The normalized factor in the first iterate is
\[
 F_d(w)=\int_0^1(1-tw^2)^{d-1}\,dt,
 \qquad e_1=V-\frac{\mathscr S\log F_d}{2d-2}.
\]
For $a=d-1\ge2$, we bound the absolute values of the binomial
coefficients to obtain
\[
 F_d-1\preceq\frac{aw^2/2}{1-aw^2/2},\qquad
 \log F_d\preceq K_a(w):=\log(1-aw^2/2)-\log(1-aw^2).
\]
Here $a(a+1)\cdots(a+m-1)/(m+1)!\le(a/2)^m$, and the logarithm series converge
on all required shifts since $a(R_1/(1-R_1))^2<1/10$.
After multiplication by $p_1-2=2(a-1)$ and evaluation at
$R_1^2=1/[32(a-1)]$, every term in the resulting majorant is nonincreasing
in $a$: the terms from $\mathscr S K_a$ are positive constants times
$a^{m-1}/(a-1)^{m+\ell-2}$, $m,\ell\ge1$.
Thus it suffices to take $a=2$. Using $\log(1+t)\le t$ gives
\begin{equation}\label{eq:power-analytic-initial}
 (p_1-2)\pnorm{e_1}{R_1}
 \le\frac1{31}+16\left(\frac{115200}{114359}-1\right)
 =\frac{17145}{114359}<\frac3{20},
 \qquad \pnorm{e_1}{R_1}<\frac3{40}.
\end{equation}
The fraction in this bound is obtained by pairing the two unit
substitutions at $R_1^2=1/32$.

Fix an iteration and write $p=p_j$, $p'=2p-2$, $R=R_j$,
$r=R_{j+1}$, $e=e_j$, and $M=\pnorm{e}{R}$.
The exact normalized recurrence is
\begin{equation}\label{eq:power-analytic-recurrence}
 e^+=\mathscr A_pe-
 \frac{\mathscr S\{\log(1+e)-e\}}{p'}-
 \frac{2V+\mathscr S\phi(pw^2(1+e))}{p'}.
\end{equation}
Assume $M\le1/8$. The norm $\pnorm{\log(1+e)-e}{R}$ is at most $\Psi(M)$.
From the product for $\sinh$ \cite[Eq.~4.36.1]{DLMF}, we obtain the
coefficient majorant
\[
 \phi(v)+v/2\preceq\frac{2v^2}{36-v^2},\qquad
 \pnorm{pw^2(1+e)}{R}\le9/64<1/6.
\]
Retaining the linear term $-pw^2(1+e)/2$, the norm of the last term in
\eqref{eq:power-analytic-recurrence} is at most
\[
 F(p)\le\frac1{p-2}
 \left(\frac1{189}+\frac{25}{384}+\frac1{64}+\frac1{200}\right)
 <\frac3{32(p-2)}\le\frac3{64}<\frac1{20}.
\]
The four terms bound $2V/p'$, the $w^2$ part, the $w^2e$ part, and the
quadratic remainder of $\phi$, respectively. Here we use
$\pnorm{\mathscr S w^2}{r}\le(25/4)r^2$, $p/p'\le2/3$, and
the operator estimate in Lemma~\ref{lem:power-scaled-difference}.
For the last bound, writing $H=\pnorm{pw^2(1+e)}{R}$,
we have
\[
 \frac{12(p-2)H^2}{p'(36-H^2)}
 \le\frac{729}{331520}<\frac1{200},
 \qquad \frac{p^2}{p'(p-2)}\le\frac43.
\]

Therefore $M^+\le a(p)M+6\Psi(M)/p'+F(p)$.
Starting from \eqref{eq:power-analytic-initial}, we prove $M<1/8$
by induction. At the first step of the induction, the error is bounded by
$(4/5)(3/40)+\Psi(3/40)+1/20<1/8$. Thereafter $p\ge6$, and the
bound is $1/16+(3/5)\Psi(1/8)+1/20=33/280<1/8$.
For $Y_j=(p_j-2)M_j$, we obtain from the same inequalities
\[
 Y_1<3/20,\qquad
 Y_2<\frac85\frac3{20}+6\Psi(3/40)+\frac3{16}<\frac12,
 \qquad Y_{j+1}\le Y_j+\frac{27}{112}<Y_j+\frac14\quad(j\ge2).
\]
Hence $Y_j<j/4$. All compositions and logarithms are analytic by the
strict disk inclusions and the preceding smallness bounds. By evenness,
$|e_j(w)|\le M_j|w|^2/R_j^2<4j|w|^2$, which proves
\eqref{eq:power-analytic-error}. On the positive ray $1+e_j>7/8$.
Since $R_{j+1}^{-1}-R_j^{-1}=(\sqrt2-1)R_j^{-1}>1$,
all predecessors lie in the
preceding exterior domain. The first two iterates are
positive by Lemma~\ref{lem:power-first}. By induction using the identity
$a_k^{[j+1]}=(a_k^{[j]})^2(1-e^{-h_j(k)})$, these curvatures agree
with those of the original sequence, whose iterates are positive
on the required integer sets. The logarithms in the normalized
recurrence are defined by their series at $1$ and $0$, so no logarithm
of $z$ on a full exterior disk is required.
\end{proof}

For $3\le d\le242$, let $J$ be the first positive integer with
$p_J\ge256$, and put $P=p_J$. Then $J\le8$ and $256\le P<510$.
For an integer $k\ge256$, set
\[
 K=\left\lceil\log_2\frac{3k^2}{P}\right\rceil,
 \qquad R_i=\frac{3k}{5}\left(\frac34\right)^i,
 \qquad D_i=\{|z-k|\le R_i\}.
\]
On $D_0$, $|z|\ge2k/5>4\sqrt{508}$.
We apply Proposition~\ref{prop:power-analytic-preparation} to obtain
\begin{equation}\label{eq:power-real-preparation}
 \left|h_J(z)-Pz^{-2}\right|\le\frac{32P}{|z|^4}.
\end{equation}
Inversion gives $k^2\Rea z^{-2}\ge175/512$ on $D_0$ and
$k^2\Rea z^{-2}\ge(1+r)^{-2}$ on a disk of radius $rk$, $r\le1/2$.
Consequently, substituting $P\ge256$ and $k\ge256$ in
\eqref{eq:power-real-preparation} yields
\[
 k^2\inf_{D_i}\Rea h_J>
 \begin{cases}80,&i=0,\\120,&i=1,\\140,&i\ge2.\end{cases}
\]

All successive radius gaps have square greater than six. For $i\le9$,
\[
 (R_i-R_{i+1})^2\ge\frac{9\cdot256^2}{400}
 \left(\frac9{16}\right)^9>8,
\]
whereas for $10\le i<K$, we use $2^iP<3k^2$ to obtain
$(R_i-R_{i+1})^2>(3P/400)(9/8)^i\ge(48/25)(9/8)^{10}>6$.
Thus $\mathcal E_i<140/k^2$ by Corollary~\ref{cor:geometric-budget}. At the first layer the sharper exact bound is
\[
 k^2\mathcal E_1=\frac2{9/400-k^{-2}}
 \le\frac{3276800}{36839}<90.
\]
The real-part inequalities in Theorem~\ref{thm:unified-entry} hold
at every iteration. At the center, we obtain from \eqref{eq:entry-center-bounds}
\[
 h_{J+K}(k)>3\left(1-\frac{32}{k^2}-\frac{140}{P}\right)
 \ge\frac{2781}{2048}>1>\log(8/3).
\]
Since $2k/5-1>4\sqrt{508}$, all required initial coefficient indices
lie in the exterior domain where the preliminary iterates are positive.
At each later step, both unit shifts lie in the preceding domain. By
the correspondence in Section~\ref{subsec:prepared-entry},
$a_k^{[j]}(d)>0$ for $0\le j\le J+K$ and $c_{J+K}(k)<3/8$.

\subsection{Finite initial ranges}

The preceding argument applies to $k\ge256$. We now verify the
inequalities at the remaining indices uniformly in the exponent. Put
\[
 I_r=[2+2^r,\min\{2+2^{r+1},242\}],\qquad0\le r\le7.
\]
Divide each $I_r$ into closed intervals of length $2^r/8$. These $63$ intervals cover $[3,242]$.

\begin{proposition}\label{ver:power-prefixes}
For $0\le r\le7$ and $d\in I_r$,
\[
 a_v^{[j]}(d)>0\quad(0\le j\le17-r,\ 1\le v\le275-j),
 \qquad c_{17-r}(k)\le13/40\quad(1\le k\le256).
\]
\end{proposition}
\begin{proof}[Computer-assisted proof]
On each parameter interval write $d=a+bx$, where $|x|\le1$ and
$a-b\ge3$. The initial ratios are
\[
 c_0(1;d)=0,\qquad c_0(k;d)=e^{-d\ell_k},\qquad
 \ell_k=-\log(1-k^{-2})\quad(k\ge2).
\]
We use polynomial enclosures of degree $32$ and coefficient and error
denominators $S=2^{512}$, retaining the same $x$ at every position and layer.
For $t=k^{-2}$, choose $m$ so that
\[
 L_m=\sum_{i=1}^m\frac{t^i}{i},\qquad
 0\le\ell_k-L_m\le\frac{t^{m+1}}{(m+1)(1-t)}=:\epsilon_m<2^{-552}.
\]
The polynomial
\[
 e^{-aL_m}\sum_{i=0}^{32}\frac{(-bL_mx)^i}{i!}
\]
has error at most
$e^{-(a-b)L_m}(bL_m)^{33}/33!+(a+b)\epsilon_m$.
To enclose each constant exponential, halve its nonnegative argument
until it is at most $1/16$, use the alternating series with error below
$2^{-552}$, and recover the original value by repeated squaring with outward rounding.
Thus every initial enclosure follows from rational arithmetic.

We apply \eqref{eq:ratio-iteration} with the product and quotient
rules of Lemma~\ref{lem:model-arithmetic}. All discarded coefficients
and rounding errors enter the residual bounds. A model $(P,E)$ on
$[-1,1]$ has upper bound $U=P(0)+\sum_{i\ge1}|[x^i]P|+E$.
Before every update, each retained upper ratio satisfies $U<1$.
In particular, $(1-U)^2$ is a positive lower bound for the corresponding
squared denominator. Nonnegativity of the ratios and positivity of the
coefficients then follow by induction from the positive initial sequence.

The initial ratios extend through index $274$, using original coefficients
through index $275$. At layer $j$ we retain $1\le k\le274-j$, preserve
the true endpoint $c_j(1)=0$, and discard the rightmost position.
By exact arithmetic, we verify $40U\le13$ for $1\le k\le256$ after
$17-r$ iterations on each whole parameter interval. Every required
predecessor is retained, and this index range overlaps
the analytic tail at $k=256$.
\end{proof}

\begin{proof}[Proof of Theorem~\ref{thm:powers}]
For $3\le d\le242$, the analytic tail estimates and
Proposition~\ref{ver:power-prefixes} give, at each $k\ge1$, a finite
integer $T=T(k)$ such that
$a_k^{[j]}(d)>0$ for $0\le j\le T$ and $c_T(k)<3/8$.
Lemma~\ref{lem:local-stop} therefore gives positivity at all orders.
For $d\ge241$, we apply Proposition~\ref{prop:large-real-powers}.

For $3\le d\le242$ and $k\ge256$, the depth in the analytic argument satisfies
\[
 T_{3/8}(k)\le J+\left\lceil\log_2\frac{3k^2}{P}\right\rceil
 =\left\lceil\log_2\frac{3k^2\,2^J}{P}\right\rceil
 \le\left\lceil\log_2\frac{3k^2}{d-2}\right\rceil,
\]
since $P=2+2^J(d-2)$. For $d\ge241$, indices $k\le\sqrt d$
have $T_{3/8}(k)=0$, while $\sqrt d<k\le\sqrt{2d-2}$ have
$T_{3/8}(k)\le1$. In the latter range the asserted ceiling is at least one.
For $k>\sqrt{2d-2}$, we similarly obtain
\[
 T_{3/8}(k)\le1+\left\lceil\log_2\frac{2k^2}{2d-2}\right\rceil
 =\left\lceil\log_2\frac{2k^2}{d-1}\right\rceil
 \le\left\lceil\log_2\frac{3k^2}{d-2}\right\rceil.
\]
For $k\le256$, at most ten iterations are needed for large exponents,
since $2\cdot256^2/480<2^9$. Together with the finite depths $17-r$,
this proves both iteration bounds.

For $d=1$,
\[
 (0,1,2,3,\ldots)\xrightarrow{\LL}(0,1,1,1,\ldots)
 \xrightarrow{\LL}(0,1,0,0,\ldots),
\]
and the last sequence is fixed.  For $d=2$, the first three iterates
begin
\[
 (0,1,7,17,31,\ldots),\quad(0,1,32,72,128,\ldots),\quad
 (0,1,952,1088,1984,\ldots).
\]
The fourth iterate at index three is
$1088^2-952\cdot1984=-705024$.
Thus the only excluded positive integer exponent is two.
\end{proof}

\section{Further results on the log-concavity operator}\label{sec:consequences}
Using Section~\ref{sec:stopping}, we derive iteration bounds from the
initial curvature and prove convergence of the normalized logarithms.

\subsection{Bounds on the number of iterations}
We relate the required number of iterations to the size of the
initial curvature and the radius of a disk on which it is holomorphic.
No particular formula for the sequence is needed.

\begin{theorem}\label{thm:curvature-scale}
Let \(\varepsilon_n>0\) tend to zero.  Suppose that the initial curvature
\(h_0\) is holomorphic on a neighborhood of
\(\lvert z-n\rvert\le R_n\), where \(0<R_n\le n/2\), and that there are
constants \(A,B>0\)
such that
\begin{equation}\label{eq:scale-hypotheses}
\Rea h_0(z)\ge A\varepsilon_n,
\qquad h_0(n)\le B\varepsilon_n,
\qquad A\varepsilon_nR_n^2\ge128.
\end{equation}
Put $K_n=\lceil\log_2(2\rho_*/(A\varepsilon_n))\rceil$.
For all sufficiently large $n$, assume that the initial coefficients
with $|k-n|\le K_n+1$ lie in the original positive support and that
$h_0(k)$ equals their initial curvature for $|k-n|\le K_n$.
Then, for all sufficiently large \(n\), $T=\tau(n)$ is finite and
\[
 a_v^{[j]}>0\quad(0\le j\le T,\ |v-n|\le T+1-j),
 \qquad c_T(n)\le q_*.
\]
Moreover,
\begin{equation}\label{eq:general-scale-time}
\tau(n)=\log_2\frac1{\varepsilon_n}+O_{A,B}(1).
\end{equation}
The conclusion is uniform over a family for which all these
hypotheses, including \(\varepsilon_n\to0\), hold uniformly.
\end{theorem}

\begin{proof}
On each symmetric disk replace $h_0$ by
$(h_0(z)+\overline{h_0(\overline z)})/2$.
This holomorphic function is real on the real diameter, agrees with the
sequence curvature at integers, and preserves the bounds in
\eqref{eq:scale-hypotheses}.  Apply
Corollary~\ref{thm:analytic-amplification} with $\mu=A\varepsilon_n$.
Since $54/R_n^2\le(27A/64)\varepsilon_n$, its iteration bounds give
\[
\left\lceil\log_2\frac{\rho_*}{(B+27A/64)\varepsilon_n}\right\rceil
\le\tau(n)\le
\left\lceil\log_2\frac{2\rho_*}{A\varepsilon_n}\right\rceil
\]
for all sufficiently large $n$. The gap and final-radius estimates
in that corollary place all required integer points in the analytic
domains. The support and agreement assumptions identify their values
with the curvatures of the original iterates, as in
Section~\ref{subsec:prepared-entry}. They also ensure that the ratio
at layer $K_n$ is admissible under essential iteration.
Since the constants depend only on the stated bounds, the same proof
applies uniformly to the family.
\end{proof}

For power-law decay, we can make these bounds more explicit. In the
inverse-square case, a separate argument is needed, since the correction and initial curvature
have the same order. Here we consider a positive sequence on a half-line;
its holomorphic curvature interpolation is real on the positive axis
and agrees with the discrete curvature at all sufficiently large integers.

The inverse-square coefficient recurrence below already occurs in
the fixed-order analysis of Hou and Li
\cite[Theorem~3.2, equation~(3.10)]{HouLi}.
We estimate $h_j(n)$ for $0\le j\le\tau(n)$, where
$h_{\tau(n)}(n)\ge\rho_*$. To bound the second differences in the
recurrence, we use the sectorial hypothesis rather than real-axis
asymptotics alone.

\begin{theorem}\label{thm:power-law}
Suppose that \(h_0\) is holomorphic for
\(\lvert z\rvert>R_0\), \(\lvert\arg z\rvert<\vartheta\), where
\(0<\vartheta<\pi/2\), and that
\begin{equation}\label{eq:power-seed}
h_0(z)=\alpha z^{-p}(1+o(1)),\qquad \alpha>0,
\end{equation}
uniformly on every closed smaller sector. If \(0<p<2\), then, for all
sufficiently large $n$, $T=\tau(n)$ is finite and
\[
 a_v^{[j]}>0\quad(0\le j\le T,\ |v-n|\le T+1-j),
 \qquad c_T(n)\le q_*.
\]
In this case,
\begin{equation}\label{eq:subcritical-time}
\tau(n)=p\log_2 n-\log_2\alpha+O(1).
\end{equation}
If \(p=2\) and \(\alpha>2\), then the same conclusion holds with
\begin{equation}\label{eq:critical-time}
\tau(n)=2\log_2 n-\log_2(\alpha-2)+O(1).
\end{equation}
The parameters and the sector in these assertions are fixed.  The
implied constants may depend on the initial function and these parameters. No
uniform bound as $\alpha\downarrow2$ is asserted in
\eqref{eq:critical-time}.
\end{theorem}

\begin{proof}
Choose $0<r<1/2$ so that $|z-n|\le rn$ lies in a smaller sector.
For $0<p<2$, the lower bound
$\Rea z^{-p}\ge\cos(p\arcsin r)(1+r)^{-p}n^{-p}>0$ implies
\eqref{eq:scale-hypotheses} with $\varepsilon_n=\alpha n^{-p}$,
since $\varepsilon_n(rn)^2\to\infty$.
Thus \eqref{eq:subcritical-time} follows from
Theorem~\ref{thm:curvature-scale}.

For $p=2$, the second-difference correction has the same order as the
initial curvature. We first increase its leading coefficient through
a fixed number of iterations. Induction on smaller sectors gives
\begin{equation}\label{eq:general-critical-drift}
 h_j(z)=\alpha_jz^{-2}+o(|z|^{-2}),\qquad
 \alpha_j=2+2^j(\alpha-2).
\end{equation}
Indeed, writing $h_j=\alpha_jz^{-2}(1+e_j)$, we have
$\D\log(1+e_j)=o(|z|^{-2})$ by Cauchy's estimate, while
$\D\log(\alpha_jz^{-2})=2z^{-2}+O(|z|^{-4})$ and
$\D\phi(h_j)=O(|z|^{-4})$.  Hence $\alpha_{j+1}=2\alpha_j-2$.
Choose a fixed $J$ with $2^J(\alpha-2)\ge1$ and
$\alpha_Jr^2/[2(1+r)^2]>128$.
By these fixed-depth asymptotics, $\Rea h_j>0$ on nested disks
for $0\le j\le J$. Their successive radius gaps exceed one, and
the last disk has radius $rn$, on which
\[
 \Rea h_J(z)\ge\frac{\alpha_J}{2(1+r)^2n^2},\qquad
 h_J(n)\le\frac{2\alpha_J}{n^2}.
\]
At each of the finitely many earlier layers the curvature at the center
tends to zero and is therefore below $\rho_*$ for large $n$. Applying
Corollary~\ref{thm:analytic-amplification} from layer $J$, we obtain
$\tau(n)=J+\log_2(n^2/\alpha_J)+O(1)$.
Since $1\le\alpha_J/[2^J(\alpha-2)]\le3$, this is
\eqref{eq:critical-time}. The asymptotics for fixed $j$ establish the
estimates through layer $J$. We use the disk estimate for
$J\le j\le\tau(n)$.
\end{proof}

\subsection{Limiting behavior of the iterates}
If $c_j(k)\le q_0<q_*$ and the neighboring ratios lie in $[0,1)$, then
$c_{j+1}(k)\le(1-q_0)^{-2}c_j(k)^2$. We use this inequality to
estimate the growth of the iterates and prove convergence of their
normalized logarithms.

\begin{theorem}\label{thm:limit-profile}
Let $a^{[j]}$ be positive on a fixed interval of integers, finite or
one-sided infinite, for every $j$. Suppose $0<q_0<q_*$ and, for each
interior $k$, there is an integer $T(k)\ge0$ with $c_{T(k)}(k)\le q_0$. Set
$C=(1-q_0)^{-2}$ and $\zeta_k=Cc_{T(k)}(k)<1$.  Then
\begin{equation}\label{eq:general-double-exp}
 c_{T(k)+r}(k)\le C^{-1}\zeta_k^{2^r}.
\end{equation}
The limit
\begin{equation}\label{eq:profile-def}
 F(k)=\lim_{j\to\infty}2^{-j}\log a^{[j]}_k
 =\log a_k+\sum_{j\ge0}2^{-j-1}\log(1-c_j(k))
\end{equation}
exists and is finite at every index of the positive support.  At each
interior index,
\begin{equation}\label{eq:profile-strict}
 2F(k)-F(k-1)-F(k+1)\ge2^{-T(k)}\log(1/\zeta_k)>0,
\end{equation}
and
\begin{equation}\label{eq:profile-error}
 0\le a^{[j]}_k e^{-2^jF(k)}-1\le\frac{c_j(k)}{1-c_j(k)}
 \qquad(j\ge T(k)).
\end{equation}
At original endpoints, $F(k)=\log a_k$ and the error is zero.
\end{theorem}
\begin{proof}
By threshold preservation, $c_{j+1}(k)\le Cc_j(k)^2$ for $j\ge T(k)$.
Since $Cq_0<1$, the ratios at this center decrease, and iteration proves
\eqref{eq:general-double-exp}. Taking logarithms of the coefficient
recurrence gives
\[
 2^{-j}\log a_k^{[j]}=\log a_k+
 \sum_{i<j}2^{-i-1}\log(1-c_i(k)).
\]
The right-hand side converges. The initial summands are finite, and
the remaining logarithms are uniformly bounded and multiplied by a
summable geometric sequence.
The omitted tail satisfies
\[
 0\le\log a_k^{[j]}-2^jF(k)
 =\sum_{i\ge j}2^{j-i-1}[-\log(1-c_i(k))]
 \le-\log(1-c_j(k)).
\]
Exponentiation proves \eqref{eq:profile-error}.  Finally
$2^{-j}[-\log c_j(k)]$ tends to the left-hand side of
\eqref{eq:profile-strict}.  Taking logarithms in
\eqref{eq:general-double-exp} gives its asserted positive lower bound.
\end{proof}

For $q_0=3/8$ the formula becomes
\[
 c_{T+r}(k)\le\frac{25}{64}
 \left(\frac{64c_T(k)}{25}\right)^{2^r}
 \le\frac{25}{64}\left(\frac{24}{25}\right)^{2^r}.
\]
The limit formula applies to power sequences with $d\ge3$,
Boros--Moll sequences, transposed Boros--Moll sequences, and the coefficient sequences of the Medina--Moll--Rowland polynomials.
A fixed normalized sequence from Euler's difference table under essential iteration has only finitely many
layers and does not fall under this theorem.

For transposed Boros--Moll sequences, $c_T(k)\le q_*$ implies
$c_T(k)<q_*$, since these ratios are rational and $q_*$ is irrational. Repeated
application of $x\mapsto(x/(1-x))^2$ brings this ratio below $3/8$
in finitely many further steps. Indeed, the map is strictly smaller
than $x$ on $(0,q_*)$, and the resulting decreasing sequence has limit
zero. Hence, at each fixed $k$, we have $c_j(k)\le3/8$ for some finite $j$.
\section*{Acknowledgments}
Matthew H.~Y.~Xie was supported by the National Natural Science Foundation
of China (Grant No.~12271403). Philip B.~Zhang was supported by the National
Natural Science Foundation of China (Grant No.~12171362) and the Natural
Science Foundation of Tianjin Municipality (Grant No.~25JCYBJC00430).

\section*{AI declaration}
Generative AI was used to develop
mathematical ideas and proof arguments, examine proposed proofs
and estimates, write and check computational code, carry out
symbolic and numerical calculations, and draft and revise the
exposition. The authors are responsible for the mathematical
content, computations, references, and exposition.

\end{document}